\documentclass[12pt]{amsart}

\usepackage{amssymb}

\usepackage{enumitem}

\usepackage{graphicx}

\makeatletter
\@namedef{subjclassname@2020}{%
  \textup{2020} Mathematics Subject Classification}
\makeatother

\usepackage[T1]{fontenc}
\newtheorem{theorem}{Theorem}[section]
\newtheorem{corollary}[theorem]{Corollary}
\newtheorem{lemma}[theorem]{Lemma}
\newtheorem{proposition}[theorem]{Proposition}

\theoremstyle{definition}

\newtheorem{remark}[theorem]{Remark}

\numberwithin{equation}{section}

\newcommand{\bfs}{\boldsymbol}
\newcommand{\A}{\mathbb A}
\newcommand{\K}{\mathbb K}
\newcommand{\N}{\mathbb N}
\newcommand{\F}{\mathbb F}
\newcommand{\Pp}{\mathbb P}

\newcommand{\fq}{\F_{\hskip-0.7mm q}}
\newcommand{\cfq}{\overline{\F}_{\hskip-0.7mm q}}

\def\ifm#1#2{\relax \ifmmode#1\else#2\fi}

\newcommand{\klk}    {\ifm {,\ldots,} {$,\ldots,$}}

\begin{document}

%%%%% To ease editing, for IMPAN journals add:

\baselineskip=17pt

%%%%%%%%%%%%%%%%

\title[Distribution of defective systems]{The distribution
of defective multivariate polynomial systems over a finite field}

\author[N. Gimenez]{Nardo Gim\'enez}
\address{Instituto del Desarrollo Humano,
Universidad Nacional de Gene\-ral Sarmiento, J.M. Guti\'errez 1150
(B1613GSX) Los Polvorines, Buenos Aires, Argentina}
\email{ngimenez@ungs.edu.ar}

\author[G. Matera]{Guillermo Matera}
\address{Instituto del Desarrollo Humano,
Universidad Nacional de Gene\-ral Sarmiento, J.M. Guti\'errez 1150
(B1613GSX) Los Polvorines, Buenos Aires, Argentina\\
National Council of Science and Technology (CONICET), Ar\-gentina}
\email{gmatera@ungs.edu.ar}

\author[M. P\'erez]{Mariana P\'erez}
\address{Universidad Nacional de Hurlingham, Instituto de Tecnolog\'ia e
Ingenier\'ia\\ Av. Gdor. Vergara 2222 (B1688GEZ), Villa Tesei,
Buenos Aires, Argentina\\
National Council of Science and Technology (CONICET), Ar\-gentina}
\email{mariana.perez@unahur.edu.ar}

\author[M. Privitelli]{Melina Privitelli}
\address{Instituto de Ciencias,
Universidad Nacional de Gene\-ral Sarmiento, J.M. Guti\'errez 1150
(B1613GSX) Los Polvorines, Buenos Aires, Argentina\\
National Council of Science and Technology (CONICET), Ar\-gentina}
\email{mprivite@ungs.edu.ar}

\date{\today}

\begin{abstract}
This paper deals with properties of the algebraic variety defined as
the set of zeros of a ``deficient'' sequence of multivariate
polynomials. We consider two types of varieties: ideal-theoretic
complete intersections and absolutely irreducible varieties. For
these types, we establish improved bounds on the dimension of the
set of deficient systems of each type over an arbitrary field. On
the other hand, we establish improved upper bounds on the number of
systems of each type over a finite field.
\end{abstract}



\maketitle

%
%---------------------------------------------------------------------
%---------------------------------------------------------------------
%---------------------------------------------------------------------
%---------------------------------------------------------------------
%---------------------------------------------------------------------
%---------------------------------------------------------------------
%---------------------------------------------------------------------
%---------------------------------------------------------------------
%
\section{Introduction}\label{section: intro}
Let $K$ be a field, $\overline{K}$ its algebraic closure, and let
$X_1,\ldots,X_r$ be indeterminates over $K$. For any $d\ge 1$, we
denote by $\mathcal{F}_d(K)$ the set of elements of
$K[X_1,\ldots,X_r]$ of degree at most $d$. Further, for $1<s<r$,
$d_1,\ldots,d_s\ge 1$ and $\bfs d_s:=(d_1,\ldots,d_s)$, we denote by
$\mathcal{F}_{\bfs d_s}(K):=\mathcal{F}_{d_1}(K)\times\cdots\times
\mathcal{F}_{d_s}(K)$ the set of $s$-tuples $\bfs
F_s:=(F_1,\ldots,F_s)$ in $K[X_1,\ldots,X_r]^s$ of degrees at most
$d_1,\ldots,d_s$ respectively. In this paper we are interested in
the geometric properties of the affine $K$-variety $V(\bfs F_s)$
defined by a given $\bfs F_s:=(F_1,\ldots,F_s)\in\mathcal{F}_{\bfs
d_s}(K)$, namely the set of common solutions with coefficients in
$\overline{K}$ of the system $F_1=0,\ldots,F_s=0$, when $\bfs F_s$
varies in $\mathcal{F}_{\bfs d_s}(K)$. In particular, we shall be
interested in the case when $K$ is the finite field $\fq$ of $q$
elements, where $q$ is a prime power. In this case, we denote
$\mathcal{F}_d:=\mathcal{F}_d(\fq)$ for $d\ge 1$ and
$\mathcal{F}_{\bfs d_s}:=\mathcal{F}_{\bfs d_s}(\fq)$ for $\bfs
d_s\in\N^s$.

For a generic $\bfs F_s\in\mathcal{F}_{\bfs d_s}(K)$, the variety
$V(\bfs F_s)\subset\overline{K}{}^r$ is an irreducible smooth
complete intersection. In fact, in the projective setting, namely
for $\bfs F_s:=(F_1,\ldots,F_s)$ varying in set $\mathcal{F}_{\bfs
d_s}^*(K)$ of $s$-tuples of homogeneous polynomials of
$K[X_0,\ldots,X_r]$ of degrees $d_1,\ldots,d_s$, the set of $\bfs
F_s$ for which $V(\bfs F_s)\subset\Pp^r$ is not a smooth complete
intersection forms a hypersurface in $\mathcal{F}_{\bfs d_s}^*(K)$,
whose degree has been precisely determined (see \cite[Theorem
1.3]{Benoist12}). This implies our assertion. Further, rather
precise estimate on the number of $\bfs F_s\in\mathcal{F}_{\bfs
d_s}^*(\fq)$ with this property are available (see \cite[Corollary
4.3]{GaMa19}).

On the other hand, for an arbitrary $\bfs F_s\in\mathcal{F}_{\bfs
d_s}(K)$, a number of ``accidents'' might happen. For example, $\bfs
F_s$ might fail to define a set-theoretic or an ideal-theoretic
complete intersection, or the variety $V(\bfs
F_s)\subset\overline{K}{}^r$ might fail to be irreducible. In
\cite[Theorems 3.2, 3.5 and 4.5]{GaMa19} it is shown that there are
hypersurfaces of $\mathcal{F}_{\bfs d_s}^*(K)$ of ``low'' degree
containing the set of homogeneous $\bfs F_s$ for which each of these
accidents occur. Further, there are estimates on the number of $\bfs
F_s\in\mathcal{F}_{\bfs d_s}^*(\fq)$ for which each of these
accidents occur.

In this paper we focuss on the set of $\bfs F_s\in\mathcal{F}_{\bfs
d_s}(K)$, for an arbitrary field $K$ or a finite field $K=\fq$, for
which two of these accidents happen: the $\bfs F_s$ which fail to
define an ideal-theoretic complete intersection, and those for which
$V(\bfs F_s)$ is not irreducible (absolutely irreducible for $\bfs
F_s\in\mathcal{F}_{\bfs d_s}$). For an arbitrary field $K$, we
provide bounds on the dimension of the set of $\bfs
F_s\in\mathcal{F}_{\bfs d_s}(\overline{K})$ for which either of
these accidents occur. These bounds are much sharper than those in
\cite{GaMa19}. Actually, we show that the set of $\bfs
F_s\in\mathcal{F}_{\bfs d_s}(\overline{K})$ for which the first
accident happens is contained in a variety of codimension at least
$r-s+2$, while those for which the second accident happens belong to
a variety of codimension at least $r-s+1$. For $K=\fq$, we provide
estimates on the number of $\bfs F_s\in\mathcal{F}_{\bfs d_s}$ for
which either of these accidents happen. We are interested in the
case $q\gg\max\{d_1,\ldots,d_s\}$. In this vein, the exponent of $q$
in our estimates reflect the improved results on the codimension of
the corresponding varieties.

A source of interest for the analysis of the frequency with which a
given $\bfs F_s\in\mathcal{F}_{\bfs d_s}(K)$ defines an
ideal-theoretic complete intersection comes from the setting of
multivariate system solving, in particular over finite fields.
Multivariate polynomial systems over finite fields arise in
connection with many fundamental problems in cryptography, coding
theory, or combinatorics; see, e.g., \cite{WoPr05}, \cite{DiGoSc06},
\cite{CaMaPr12}, \cite{CeMaPePr14}, \cite{MaPePr14},
\cite{MaPePr16}. Many times in these settings it is necessary to
``solve'' a given system, where solving may mean to find one,
several, or all the corresponding solutions with coefficients in
finite field under consideration.

There are two main families of methods for solving multivariate
systems over finite fields. On one hand, we have the so-called
``rewriting'' methods, particularly Gr\"obner basis methods (see,
e.g., \cite{Faugere99} or \cite{Faugere02}) and variants such as XL
and MXL (see, e.g., \cite{CoKlPaSh00} or \cite{MoCaDiBuBu09}). On
the other hand, we have ``geometric'' methods, which compute a
suitable description of the variety defined by the system under
consideration (see, e.g., \cite{HuWo99}, \cite{CaMa06a} or
\cite{HoLe21}). It is a crucial point from the complexity point of
view that both rewriting and geometric methods behave well on
systems defining an ideal-theoretic complete intersection (see,
e.g., \cite{BaFaSa15} or \cite[Chapter 26]{BoChGiLeLeSaSc17} for
Gr\"obner bases and \cite{CaMa06a} or \cite{HoLe21} for geometric
solvers). Therefore, information on the frequency with which such
kind of systems occur is critical in this setting. Concerning this
question, we have the following result (see Theorem \ref{th:
dimension F_s not red reg seq} and Corollary \ref{coro: number
systems not H}).
\begin{theorem}\label{th: F_s not red reg seq}
Let $K$ be a field and $\overline{K}$ its algebraic closure. For
$\bfs d_s:= (d_1,\ldots,d_s)\in\N^s$ with $d_i\ge 2$ for $1\le i\le
s$, denote by $B_1\subset \mathcal{F}_{\bfs d_s}(\overline{K})$ the
set of $\bfs F_s$ which do not define an ideal-theoretic complete
intersection. Then
\begin{equation}\label{eq: bound dim B_1}
\dim B_1\le \dim \mathcal{F}_{\bfs d_s}(\overline{K})-r+s-2.
\end{equation}
Further, for $K=\fq$, let $B_1(\fq)=B_1\cap \mathcal{F}_{\bfs d_s}$,
$\delta:=d_1\cdots d_s$ and $\sigma:=d_1+\cdots +d_s -s$.
Considering the uniform probability in $\mathcal{F}_{\bfs d_s}$, the
probability ${\sf P}_1$ of $B_1(\fq)$ is bounded in the following
way:
\begin{equation}\label{eq: bound cardinality B_1(Fq)}
{\sf P}_1\le \bigg(\frac{2\,s\,\sigma\,\delta}{q}\bigg)^{r-s+2}.
\end{equation}
\end{theorem}

For perspective, in \cite[Theorem 3.5]{GaMa19} and \cite[Corollary
3.7]{GaMa19} it is shown that the corresponding set $B_1^*\subset
\mathcal{F}_{\bfs d_s}^*(\overline{K})$ satisfies $\dim B_1^*\le
\dim \mathcal{F}_{\bfs d_s}^*(\overline{K})-1$ and the corresponding
probability ${\sf P}_1^*$ is bounded in the following way:
$${\sf P}_1^*\le\frac{2\,s\,\sigma\,\delta}{q}.
$$
The bounds \eqref{eq: bound dim B_1} and \eqref{eq: bound
cardinality B_1(Fq)} represent significant improvements of these
results.

Concerning the second accident, for $K=\fq$ it is well-known that,
if the variety $V(\bfs F_s)\subset\cfq{\!}^r$ defined by $\bfs
F_s\in\mathcal{F}_{\bfs d_s}$ is an absolutely irreducible complete
intersection, then there are good estimates on the deviation from
the expected number of points of $V(\bfs F_s)$ in $\fq^r$ (see,
e.g., \cite{GhLa02a}, \cite{CaMa06}). This motivates the study of
the ``frequency'' with which such a geometric property arises.

For $s=1$, the variety defined by a single polynomial $F_1\in
\fq[X_1\klk X_r]$ is a hypersurface, which is absolutely irreducible
if the polynomial $F_1$ is. Counting irreducible multivariate
polynomials over a finite field is a classical subject which goes
back to the works of \cite{Carlitz63}, \cite{Carlitz65} and
\cite{Cohen68}; see \cite[Section 3.6]{MuPa13} for further
references. In \cite{GaViZi13}, exact formulas on the number of
absolutely irreducible multivariate polynomials over a finite field
and easy--to--use approximations are provided. The first result on
the number of sequences of polynomials $F_1\klk F_s$ over a finite
field defining an absolutely irreducible projective variety are due
to \cite{GaMa19}. We have the following result (see Theorem \ref{th:
dimension F_s not abs irred} and Corollary \ref{coro: number systems
not AI}).
\begin{theorem}\label{th: F_s not AI}
Let $K$ be a field and let $\overline{K}$ be its algebraic closure.
For $\bfs d_s\in\N^s$, denote by $B_2\subset \mathcal{F}_{\bfs
d_s}(\overline{K})$ the set of $\bfs F_s\in \mathcal{F}_{\bfs
d_s}(\overline{K})$ such that $V(\bfs F_s)$ is not an irreducible
complete intersection. Then
\begin{equation}\label{eq: bound dim B_2}
\dim B_2\le \dim \mathcal{F}_{\bfs d_s}(\overline{K})-r+s-1.
\end{equation}
Further, for $K=\fq$ and $\bfs d_s:= (d_1,\ldots,d_s)\in\N^s$ with
$d_i\ge 2$ for $1\le i\le s$, let $B_2(\fq)=B_2\cap
\mathcal{F}_{\bfs d_s}$, $\delta:=d_1\cdots d_s$ and
$\sigma:=d_1+\cdots +d_s -s$. Considering the uniform probability in
$\mathcal{F}_{\bfs d}^s$, the probability ${\sf P}_2$ of $B_2(\fq)$
is bounded in the following way:
\begin{equation}\label{eq: bound cardinality B_2(Fq)}
{\sf P}_2\le \bigg(\frac{2\,s\,\sigma^2\delta}{q}\bigg)^{r-s+1}.
\end{equation}
\end{theorem}

In \cite[Theorem 4.5]{GaMa19} and \cite[Corollary 4.6]{GaMa19} it is
shown that the corresponding set $B_2^*\subset \mathcal{F}_{\bfs
d_s}^*(\overline{K})$ satisfies $\dim B_2^*\le \dim
\mathcal{F}_{\bfs d_s}^*(\overline{K})-1$ and the probability ${\sf
P}_2^*$ is bounded in the following way:
$${\sf P}_2^*\le\frac{3\,s\,\sigma^2\delta}{q}.
$$
The bounds \eqref{eq: bound dim B_2} and \eqref{eq: bound
cardinality B_2(Fq)} significantly improve these results.

The paper is organized as follows. In Section \ref{section:
notation, notations} we briefly recall the notions and notations of
algebraic geometry and finite fields we use. In Section
\ref{section: dimension defective systems} we deal with an arbitrary
field $K$, and establish the results on the dimension of the set of
defective systems of Theorems \ref{th: F_s not red reg seq} and
\ref{th: F_s not AI}. Then, in Section \ref{section: number
defective systems} we deal with a finite field $\fq$ and establish
the results on the number of defective systems of Theorems \ref{th:
F_s not red reg seq} and \ref{th: F_s not AI}.

%
%----------------------------------------------------------------------------
%----------------------------------------------------------------------------
%----------------------------------------------------------------------------
%----------------------------------------------------------------------------
%----------------------------------------------------------------------------
%----------------------------------------------------------------------------
%----------------------------------------------------------------------------
%----------------------------------------------------------------------------
%
\section{Preliminaries}
\label{section: notation, notations}
We use standard notions and notations of commutative algebra and
algebraic geometry as can be found in, e.g., \cite{Harris92},
\cite{Kunz85} or \cite{Shafarevich94}.

Let $K$ be field and $\overline{K}$ its algebraic closure. We denote
by $\A^r:=\A^r(\overline{K})$ the $r$--dimensional affine space
$\overline{K}{}^r$ and by $\Pp^r:=\Pp^r(\overline{K})$ the
$r$--dimensional projective space over $\overline{K}$. By a {\em
projective variety defined over} $K$ (or a projective $K$--variety
for short) we mean a subset $V\subset \Pp^r$ of common zeros of
homogeneous polynomials $F_1,\ldots, F_m \in K[X_0 ,\ldots, X_r]$.
Correspondingly, an {\em affine variety of $\A^r$ defined over} $K$
(or an affine $K$--variety) is the set of common zeros in $\A^r$ of
polynomials $F_1,\ldots, F_{m} \in
K[X_1,\ldots, X_r]$. %We think a projective or affine
%$K$--variety to be equipped with the induced Zariski topology.
We shall frequently denote by $V(\bfs F_m)=V(F_1\klk F_m)$ or
$\{\bfs F_m=\bfs 0\}=\{F_1=0\klk F_m=0\}$ the affine or projective
$K$--variety consisting of the common zeros of the polynomials $\bfs
F_m:=(F_1\klk F_m)$.

In what follows, unless otherwise stated, all results referring to
varieties in general should be understood as valid for both
projective and affine varieties. A $K$--variety $V$ is $K$--{\em
irreducible} if it cannot be expressed as a finite union of proper
$K$--subvarieties of $V$. Further, $V$ is {\em absolutely
irreducible} if it is $\overline{K}$--irreducible. Any $K$--variety
$V$ can be expressed as an irredundant union $V=\mathcal{C}_1\cup
\cdots\cup\mathcal{C}_s$ of irreducible (absolutely irreducible)
$K$--varieties, unique up to reordering, which are called the {\em
irreducible} ({\em absolutely irreducible}) $K$--{\em components} of
$V$.

For a $K$--variety $V$ contained in $\Pp^r$ or $\A^r$, we denote by
$I(V)$ its {\em defining ideal}, namely the set of polynomials of
$K[X_0,\ldots, X_r]$, or of $K[X_1,\ldots, X_r]$, vanishing on $V$.
The {\em coordinate ring} $K[V]$ of $V$ is the quotient ring
\linebreak $K[X_0,\ldots,X_r]/I(V)$ or $K[X_1,\ldots,X_r]/I(V)$. The
{\em dimension} $\dim V$ of $V$ is the length $n$ of the longest
chain $V_0\varsubsetneq V_1 \varsubsetneq\cdots \varsubsetneq V_n$
of nonempty irreducible $K$--varieties contained in $V$. We say that
$V$ has {\em pure dimension} $n$ (or simply it is {\em
equidimensional}) if all the irreducible $K$--components of $V$ are
of dimension $n$.

A $K$--variety in $\Pp^r$ or $\A^r$ of pure dimension $r-1$ is
called a $K$--{\em hypersurface}. A $K$--hypersurface in $\Pp^r$ (or
$\A^r$) is the set of zeros of a single nonzero polynomial of
$K[X_0\klk X_r]$ (or of $K[X_1\klk X_r]$).

The {\em degree} $\deg V$ of an irreducible $K$-variety $V$ is the
maximum number of points lying in the intersection of $V$ with a
linear space $L$ of codimension $\dim V$, for which $V\cap L$ is a
finite set. More generally, following \cite{Heintz83} (see also
\cite{Fulton84}), if $V=\mathcal{C}_1\cup\cdots\cup \mathcal{C}_s$
is the decomposition of $V$ into irreducible $K$--components, we
define the degree of $V$ as
$$\deg V:=\sum_{i=1}^s\deg \mathcal{C}_i.$$
We shall use the following {\em B\'ezout inequality} (see
\cite{Heintz83}, \cite{Fulton84}, \cite{Vogel84}): if $V$ and $W$
are $K$--varieties of the same ambient space, then
\begin{equation}\label{eq: Bezout}
\deg (V\cap W)\le \deg V \cdot \deg W.
\end{equation}

Let $V\subset\A^r$ be a $K$--variety and $I(V)\subset K[X_1,\ldots,
X_r]$ its defining ideal. Let $\bfs x$ be a point of $V$. The {\em
dimension} $\dim_{\bfs x}V$ {\em of} $V$ {\em at} $\bfs x$ is the
maximum of the dimensions of the irreducible $K$--components of $V$
that contain $\bfs x$. If $I(V)=(F_1,\ldots, F_m)$, the {\em tangent
space} $T_{\bfs x}V$ {\em to $V$ at $\bfs x$} is the kernel of the
Jacobian matrix $(\partial F_i/\partial X_j)_{1\le i\le m,1\le j\le
r}(\bfs x)$ of the polynomials $F_1,\ldots, F_m$ with respect to
$X_1,\ldots, X_r$ at $\bfs x$. We have (see, e.g., \cite[page
94]{Shafarevich94})
$$\dim T_{\bfs x}V\ge \dim_{\bfs x}V.$$
The point $\bfs x$ is {\em regular} if $\dim T_{\bfs x}V=\dim_xV$.
Otherwise, the point $x$ is called {\em singular}. The set of
singular points of $V$ is the {\em singular locus}
$\mathrm{Sing}(V)$ of $V$; a variety is called {\em nonsingular} if
its singular locus is empty. For a projective variety, the concepts
of tangent space, regular and singular point can be defined by
considering an affine neighborhood of the point under consideration.
%
%----------------------------------------------------------------------------
%----------------------------------------------------------------------------
%----------------------------------------------------------------------------
%----------------------------------------------------------------------------
%
\subsection{Regular sequences}
\label{subsec: reduced reg sequences} Let $K$ be a field. A
\emph{set-theoretic complete intersection} is a variety $V(F_1 \klk
F_s) \subseteq \A^r$ defined by $s\le r$ polynomials $F_1 \klk
F_s\in K[X_1 \klk X_r]$ which is of pure dimension $r-s$. If $s\le
r+1$ homogeneous polynomials $F_1 \klk F_s$ in $K[X_0 \klk X_r]$
define a projective $K$-variety $V(F_1 \klk F_s)\subseteq
\mathbb{P}^r$ which is of pure dimension $r-s$ (the case $r-s=-1$
meaning  that $V(F_1,\ldots,F_s)$ is the empty set), then the
variety is called a set-theoretic complete intersection. Elements
$F_1 \klk F_s\in K[X_1 \klk X_r]$ form a \emph{regular sequence} if
the ideal $(F_1 \klk F_s)$ they define in $K[X_1 \klk X_r]$ is a
proper ideal, $F_1$ is nonzero and, for $2\le i \le s$, $F_i$ is
neither zero nor a zero divisor in $K[X_1 \klk X_r]/(F_1 \klk
F_{i-1})$. If in addition $(F_1 \klk F_s)$ is a radical ideal of
$K[X_1 \klk X_r]$, then we say that $V(F_1 \klk F_s)$ is an
\emph{ideal-theoretic complete intersection}. In what follows we
shall use the following result (see, e.g., \cite[Chapitre 3,
Remarque 2.2]{Lejeune84}).
\begin{lemma}\label{lemma: reg seq and complete int}
If $F_1 \klk F_s$ are  $s\le r+1$ homogeneous polynomials in $K[X_0
\klk X_r]\setminus K$, the following conditions  are equivalent:
   \begin{itemize}
     \item $F_1 \klk F_s$ define a set-theoretic complete intersection projective variety of $\mathbb{P}^r$.
     \item $F_1 \klk F_s$  is a regular sequence in $K[X_0 \klk X_r]$.
   \end{itemize}
\end{lemma}
%
%----------------------------------------------------------------
%----------------------------------------------------------------
%----------------------------------------------------------------
%----------------------------------------------------------------
%
\subsection{Rational points}
We denote by $\A^r(\fq)$ the $n$--dimensional $\fq$--vector space
$\fq^r$. For an affine variety $V\subset\A^r$, we denote by $V(\fq)$
the set of $\fq$--rational points of $V$, namely $V(\fq):=V\cap
\A^r(\fq)$ in the affine case. For an affine variety $V$ of
dimension $n$ and degree $\delta$ we have the upper bound (see,
e.g., \cite[Lemma 2.1]{CaMa06})
\begin{equation}\label{eq: upper bound -- affine gral}
   |V(\fq)|\leq \delta q^n.
\end{equation}
%
%----------------------------------------------------------------------------
%----------------------------------------------------------------------------
%----------------------------------------------------------------------------
%----------------------------------------------------------------------------
%----------------------------------------------------------------------------
%----------------------------------------------------------------------------
%----------------------------------------------------------------------------
%----------------------------------------------------------------------------
%
\section{The dimension of the set of defective systems}
\label{section: dimension defective systems}
For $1<s<r$ and $\bfs d_s:=(d_1,\ldots,d_s)\in\N^s$, we recall that
$\mathcal{F}_{\bfs d_s}(\overline{K})$ denotes the set of $s$-tuples
$\bfs F_s:=(F_1, \dots, F_s)$ with $F_i\in\mathcal{F}_{d_i}:=\{F\in
\overline{K}[X_1, \dots, X_r]:\deg F \le d_i\}$ for $1\le i\le s$,
where $K$ is an arbitrary field.

In this section we show that there are ``few'' $\bfs
F_s\in\mathcal{F}_{\bfs d_s}(\overline{K})$ which either do not
define an ideal-theoretic complete intersection or the variety
$V(\bfs F_s)$ is not irreducible. More precisely, for each of these
``deficiencies'', we shall show that the set of deficient systems in
contained in a variety of positive codimension of $\mathcal{F}_{\bfs
d_s}(\overline{K})$. For this purpose, we shall consider the
following incidence variety
\begin{equation}\label{eq: incidence var W}
W:=\{(\bfs F_s,\bfs x)\in\mathcal{F}_{\bfs d_s}(\overline{K})\times
\Pp^r(\overline{K}):\bfs F_s^h(\bfs x)=0, \Delta_i(\bfs x)=0\ (1\le
i\le N)\},
\end{equation}
where $\bfs F_s^h$ denote the homogenization of $\bfs F_s$ with
homogenizing variable $X_0$ and $\Delta_1,\ldots,\Delta_N$ denote
the maximal minors of the Jacobian matrix of $\bfs F_s^h$ with
respect to $X_0,\ldots,X_r$.

A number of remarks are in order. First, in the proof of \cite[Lemma
3.2]{Benoist12} it is shown the following assertion.
\begin{remark}
$W$ is irreducible of dimension $\dim \mathcal{F}_{\bfs
d_s}(\overline{K})-1$.
\end{remark}

Let $\pi:W\to \mathcal{F}_{\bfs d_s}(\overline{K})$ be the
projection on the first coordinate. By \cite[Lemma 3.2, Corollary
3.3 and Proposition 4.2]{Benoist12} we deduce the following remark.
\begin{remark}
$\pi(W)$ is an irreducible hypersurface of $\mathcal{F}_{\bfs
d_s}(\overline{K})$.
\end{remark}

%Finally, from, e.g., \cite[Theorem 3.12]{Harris92} we have the
%following remark.
%%
%\begin{remark}\label{rem: pi is closed}
%$\pi$ is closed.
%\end{remark}
%
%----------------------------------------------------------------------------
%----------------------------------------------------------------------------
%----------------------------------------------------------------------------
%----------------------------------------------------------------------------
%
\subsection{Systems not defining an ideal-theoretic complete intersection}
\label{subsec: dimension systems not H}
First we analyze the set of systems which do not define an
ideal-theoretic complete intersection. More precisely, denote
\begin{align}
A_1&:=\{\bfs F_s\in \mathcal{F}_{\bfs d_s}(\overline{K}):\bfs F_s\textrm{ defines an ideal-theoretic complete intersection}\}, \notag\\
B_1&:=\mathcal{F}_{\bfs d_s}(\overline{K})\setminus A_1.
\label{def: B_1}\end{align}
%
%where $\deg (F_1,\ldots,F_s):=(\deg F_1,\ldots,\deg F_s)$.

Denote by $W_{r-s}$ the set of $(\bfs F_s,\bfs x)\in W$ such that
$\pi^{-1}(\bfs F_s)$ has dimension at least $r-s$.  According to
\cite[\S I.6.3, Corollary]{Shafarevich94}, $\pi(W_{r-s})$ is closed,
%Then Remark \ref{rem: pi is closed} shows
%that $\pi(W_{r-s})$ is closed,
and the theorem on the dimension of fibers (see, e.g., \cite[\S
I.6.3, Theorem 7]{Shafarevich94}) implies that
\begin{equation}\label{eq: upper bound W_(r-s)}
\dim \mathcal{F}_{\bfs d_s}(\overline{K})-2\ge \dim W_{r-s}\ge\dim
\pi (W_{r-s})+(r-s).
\end{equation}

We have the following result.
\begin{lemma}\label{lemma: homogenized system is red reg seq}
For any $\bfs F_s\in \mathcal{F}_{\bfs d_s}(\overline{K})\setminus
\pi (W_{r-s})$, we have the following assertions:
\begin{itemize}
  \item $V(\bfs F_s^h)$ is of pure dimension $r-s$.
  \item The ideal $I(\bfs F_s^h)$ is radical.
\end{itemize}
\end{lemma}
\begin{proof}
Observe that each irreducible component of $V(\bfs F_s^h)$ has
dimension at least $r-s$. By the definition of $W_{r-s}$, the set of
points of $V(\bfs F_s^h)$ for which the Jacobian matrix of $\bfs
F_s^h$ is not of full rank, has dimension at most $r-s-1$. This
implies that the Jacobian matrix of $\bfs F_s^h$ at a generic point
of any irreducible component of $V(\bfs F_s^h)$ is of full rank. As
a consequence, all the irreducible components of $V(\bfs F_s^h)$
have dimension $r-s$, which proves the first assertion. %It follows
%that $\dim V(\bfs F_s^h)=r-s$, and thus $V(\bfs F_s^h)$ is of pure
%dimension $r-s$.

Further, as the set of points of $V(\bfs F_s^h)$ for which the
Jacobian matrix of $\bfs F_s^h$ is not of full rank, has codimension
at least $1$ in $V(\bfs F_s^h)$, the second assertion is a
consequence of \cite[Theorem 18.15]{Eisenbud95}.
\end{proof}

To establish our bound on the dimension of $B_1$, we need an upper
bound on the dimension of the set of $\bfs F_s$ for which $V(\bfs
F_s)$ is empty. The bound in our next result is probably not
optimal, but sufficient for our purposes.
\begin{lemma}\label{lemma: dimension B_0}
Suppose that $d_i\ge 2$ for $1\le i\le s$. Let
$B_0\subset\mathcal{F}_{\bfs d_s}(\overline{K})$ be the set of $\bfs
F_s:=(F_1,\ldots,F_s)$ such that $\deg F_i=d_i$ for $1\le i\le s$
and $V(\bfs F_s)\subset\A^r(\overline{K})$ is empty. Then $\dim
B_0\le \dim \mathcal{F}_{\bfs d_s}(\overline{K})-r+s-2$.
\end{lemma}
\begin{proof}
Let $\bfs F_s:=(F_1,\ldots,F_s)\in \mathcal{F}_{\bfs
d_s}(\overline{K})$ be such that $\deg F_i=d_i$ for $1\le i\le s$.
Denote by $F_i^{\mathrm{in}}$ the homogeneous component of $F_i$ of
degree $d_i$ for $1\le i\le s$ and $\bfs
F_s^{\mathrm{in}}:=(F_1^{\mathrm{in}},\ldots,F_s^{\mathrm{in}})$. By
hypothesis, $F_i^{\mathrm{in}}$ is nonzero for $1\le i\le s$. If
$V(\bfs F_s^{\mathrm{in}})\subset\Pp^{r-1}$ satisfies the condition
$\dim V(\bfs F_s^{\mathrm{in}})=r-1-s$, then, according to Lemma
\ref{lemma: reg seq and complete int}, the polynomials $\bfs
F_s^{\mathrm{in}}$ form a regular sequence. By, e.g., \cite[Lemma
5.4]{MaPePr20}, the polynomials $\bfs F_s$ also form a regular
sequence, and thus $\dim V(\bfs F_s)=r-s$. As a consequence, $B_0$
is contained in the set of $\bfs F_s:=(F_1,\ldots,F_s)$ such that
$\deg F_i=d_i$ for $1\le i\le s$ and $\dim V(\bfs
F_s^{\mathrm{in}})>r-1-s$.

Let $\bfs F_s\in\mathcal{F}_{\bfs d_s}(\overline{K})$ be such that
$\dim V(\bfs F_s^{\mathrm{in}})>r-1-s$, and let $j\le s-1$ be
maximum such that $\dim
V(F_1^{\mathrm{in}},\ldots,F_j^{\mathrm{in}})=r-1-j$. Observe that
$V(F_1^{\mathrm{in}},\ldots,F_j^{\mathrm{in}})$ is of pure dimension
$r-1-j$. As $\dim
V(F_1^{\mathrm{in}},\ldots,F_j^{\mathrm{in}})=V(F_1^{\mathrm{in}},
\ldots,F_{j+1}^{\mathrm{in}})$, the hypersurface defined by
$F_{j+1}^{\mathrm{in}}$ must contain an irreducible component of
$V(F_1^{\mathrm{in}},\ldots,F_j^{\mathrm{in}})$. By \cite[Lemme
3.3]{Benoist11}, the set of homogeneous polynomials of degree
$d_{j+1}$ of $\overline{K}[X_1,\ldots,X_r]$ containing an
irreducible component of
$V(F_1^{\mathrm{in}},\ldots,F_j^{\mathrm{in}})$ has codimension at
least
$$\binom{d_{j+1}+r-1-j}{d_{j+1}}\ge
\binom{d_{j+1}+r-s}{d_{j+1}}\ge \binom{r-s+2}{2}\ge r-s+2.$$

Now for $1\le j\le s-1$, denote by $B_j^{\mathrm{in}}$ the set of
$\bfs F_s\in\mathcal{F}_{\bfs d_s}(\overline{K})$ such that $\deg
F_i=d_i$ for $1\le i\le s$, $\dim V(F_1^{\mathrm{in}},\ldots,
F_i^{\mathrm{in}})=r-1-i$ for $1\le i\le j$ and $\dim
V(F_1^{\mathrm{in}},\ldots,F_{j+1}^{\mathrm{in}})=r-1-j$. Our
previous arguments show that
$B_0\subset\bigcup_{j=1}^{s-1}B_j^{\mathrm{in}}$. Denote by
$\pi_j:B_j^{\mathrm{in}}\to\mathcal{F}_{d_1}\times\cdots\times\mathcal{F}_{d_j}$
the projection defined by $\pi_j(F_1,\ldots,F_s):=(F_1,\ldots,F_j)$.
Then each nonempty fiber of $\pi_j$ has dimension at most
$\sum_{i=j+1}^s\dim \mathcal{F}_{d_i}(\overline{K})-(r-s+2)$. By the
theorem on the dimension of fibers we conclude that
$$\dim B_j^{\mathrm{in}}-\sum_{i=1}^j\dim
\mathcal{F}_{d_i}(\overline{K})\!\le \dim B_j^{\mathrm{in}}-\dim
\pi(B_j^{\mathrm{in}})\!\le\!\!\! \sum_{i=j+1}^s\!\dim
\mathcal{F}_{d_i}(\overline{K})-(r-s+2).$$
The conclusion of the lemma readily follows.
\end{proof}
\begin{remark}\label{rem: dimension B_0 for d_i=1}
It is easy to see that, if $d_i=1$ for some of the $d_i$, then the
proof of Lemma \ref{lemma: dimension B_0} shows that $\dim B_0\le
\dim \mathcal{F}_{\bfs d_s}(\overline{K})-r+s-1$.
\end{remark}

Now we are able to establish our result on the dimension of $B_1$.
\begin{theorem}\label{th: dimension F_s not red reg seq}
If $d_i\ge 2$ for $1\le i\le s$, then $\dim B_1\le \dim
\mathcal{F}_{\bfs d_s}(\overline{K})-r+s-2$.
\end{theorem}
\begin{proof}
Let $\bfs F_s:=(F_1,\ldots,F_s)\notin\pi(W_{r-s})\cup B_0$ be such
that $\deg F_i=d_i$ for $1\le i\le s$. We claim that $\bfs F_s$
define an ideal-theoretic complete intersection.

Indeed, according to Lemma \ref{lemma: homogenized system is red reg
seq}, $V(\bfs F_s^h)$ is of pure dimension $r-s$. As it is defined
by $s$ polynomials, by Lemma \ref{lemma: reg seq and complete int}
we conclude that $\bfs F_s^h$ forms a regular sequence, which in
turn implies that $V(F_1^h,\ldots,F_j^h)$ is of pure dimension $r-j$
for $1\le j\le s$. As a consequence, $V(F_1,\ldots,F_j)$ must be of
dimension at most $r-j$ for $1\le j\le s$, since otherwise we would
have $\dim V(F_1^h,\ldots,F_j^h)>r-j$. Further, as it is defined by
$j$ polynomials, $V(F_1,\ldots,F_j)$ must be of pure dimension $r-j$
or empty for $1\le j\le s$, and the condition $\bfs F_s\notin B_0$
implies that $V(F_1,\ldots,F_j)$ is of pure dimension $r-j$ for
$1\le j\le s$. This proves that $\bfs F_s$ forms a regular sequence.

Next, let $F$ be any element in the radical ideal of $I(\bfs F_s)$.
A straightforward computation shows that $x_0F^h$ belongs to the
radical of $I(\bfs F_s^h)$. As $I(\bfs F_s^h)$ is radical by Lemma
\ref{lemma: homogenized system is red reg seq}, we conclude that
$x_0F^h \in I(\bfs F_s^h)$. Specializing $x_0 \mapsto 1$ we deduce
that $F \in I(\bfs F_s)$, which proves that $I(\bfs F_s)$ is a
radical ideal.

%We claim that the Jacobian matrix of $\bfs F_s$ at a generic point
%of any irreducible component $\mathcal{C}$ of $V(\bfs F_s)$ is of
%full rank. Indeed, considering the standard embedding of the affine
%space $\A^r$ in the projective space $\Pp^r$, the image of a generic
%point of an irreducible component $\mathcal{C}$ of $V(\bfs F_s)$ is
%a generic point of the projective closure of $\mathcal{C}$, which is
%an irreducible component of $V(\bfs F_s^h)$ not contained in the
%hyperplane at infinity. As the ideal $I(\bfs F_s^h)$ is radical, the
%Jacobian matrix of $\bfs F_s^h$ has maximal rank at such a point,
%which implies our claim.
%
%It follows that the set of points of $V(\bfs F_s)$ for which the
%Jacobian matrix of $\bfs F_s$ is not of full rank, has codimension
%at least $1$ in $V(\bfs F_s)$, and the second assertion is a
%consequence of \cite[Theorem 18.15]{Eisenbud95}.

As a consequence, any $\bfs F_s\in B_1$ is either contained in $\pi
(W_{r-s})\cup B_0$, which is a subvariety of $\mathcal{F}_{\bfs
d_s}(\overline{K})$ of codimension at least $r-s+2$, or it is
defined by polynomials $F_1,\ldots,F_s$ with $\deg F_i<d_i$ for some
$i$.
%Arguing in this way for $1\le j\le s$, we conclude that the set of
%$\bfs F_s:=(F_1,\ldots,F_s)$ with $\deg F_i=d_i$ for $1\le i\le s$
%such that $F_1,\ldots,F_j$ is a reduced regular sequence is
%contained in a subvariety of $\mathcal{F}_{\bfs d_s}(\cfq)$ of
%codimension at least $r+2-j$.
%It follows that the set of $\bfs F_s$ which do not form a reduced
%regular sequence is contained in a subvariety of $\mathcal{F}_{\bfs
%d_s}(\overline{K})$ of codimension at least $r-s+2$.
We conclude that $B_1$ is contained in a finite union of
subvarieties of $\mathcal{F}_{\bfs d_s}(\overline{K})$ of
codimension at least $r-s+2$, showing thus the theorem.
\end{proof}
\begin{remark}\label{rem: expression B_1}
According to the last paragraph of the proof of Theorem \ref{th:
dimension F_s not red reg seq}, the set $B_1$ in contained in the
union
$$
B_1\subset  \bigcup_{i=1}^s L_i  \cup \pi(W_{r-s})\cup B_0,
$$
where $L_i:=\mathcal{F}_{d_1}(\overline{K})\times\cdots
\times\mathcal{F}_{d_{i-1}}(\overline{K})\times
\mathcal{F}_{d_i-1}(\overline{K})\times\mathcal{F}_{d_{i+1}}(\overline{K})
\times\cdots\times \mathcal{F}_{d_s}(\overline{K})$ is the set of
$\bfs F_s:=(F_1,\ldots,F_s)\in\mathcal{F}_{\bfs d_s}(\overline{K})$
with $\deg F_i<d_i$ for $1\le i\le s$, and $B_0$ is defined in Lemma
\ref{lemma: dimension B_0}.
\end{remark}
%
%----------------------------------------------------------------------------
%----------------------------------------------------------------------------
%----------------------------------------------------------------------------
%----------------------------------------------------------------------------
%
\subsection{Systems not defining an irreducible variety}
Next we consider the set of systems not defining an irreducible
variety. More precisely, denote
\begin{equation}
A_2:=\{\bfs F_s\in A_1:\,V(\bfs F_s)\text{ is irreducible}\}, \quad
B_2:=\mathcal{F}_{\bfs d_s}(\overline{K})\setminus A_2. \label{def:
B}\end{equation}
Our aim is to establish an upper bound on the dimension of $B_2$.

Let $W_{r-s-1}$ be the set of $(\bfs F_s,\bfs x)\in W$ such that
$\pi^{-1}(\bfs F_s)$ has dimension at least $r-s-1$. \cite[\S I.6.3,
Corollary]{Shafarevich94} implies that $\pi(W_{r-s-1})$ is closed,
%subvariety of $W$. Then Remark \ref{rem: pi is closed} shows that
%$\pi(W_{r-s-1})$ is closed,
and by the theorem on the dimension of fibers (see, e.g., \cite[\S
I.6.3, Theorem 7]{Shafarevich94}) we conclude that
\begin{equation}\label{eq: upper bound W_(r-s+1)}
\dim \mathcal{F}_{\bfs d_s}(\overline{K})-2\ge \dim W_{r-s-1}\ge\dim
\pi (W_{r-s-1})+(r-s-1).
\end{equation}

We have the following result.
\begin{lemma}
For any $\bfs F_s\in \mathcal{F}_{\bfs d_s}(\overline{K})\setminus
\pi (W_{r-s-1})$, we have the following assertions:
\begin{itemize}
  \item $V(\bfs F_s^h)$ is of pure dimension $r-s$.
  \item $I(\bfs F_s^h)$ is radical ideal.
  \item $V(\bfs F_s^h)$ is irreducible.
\end{itemize}
\end{lemma}
\begin{proof}
Since $\mathcal{F}_{\bfs d_s}(\overline{K})\setminus \pi
(W_{r-s-1})\subset \mathcal{F}_{\bfs d_s}(\overline{K})\setminus \pi
(W_{r-s})$, the first and the second assertions are consequences of
Lemma \ref{lemma: homogenized system is red reg seq}. Further, the
set of singular points of $V(\bfs F_s^h)$ has dimension at most
$r-s-2$, which implies that $V(\bfs F_s^h)$  is a normal variety,
and thus irreducible (see, e.g., \cite[Fact 2.1]{GaMa19}), which
completes the proof.
\end{proof}

Now we prove our result on the dimension of $B_2$.
\begin{theorem}\label{th: dimension F_s not abs irred}
$\dim B_2\le \dim \mathcal{F}_{\bfs d_s}(\overline{K})-r+s-1$.
\end{theorem}
\begin{proof}
Let $\bfs F_s=(F_1,\ldots,F_s)\notin\pi(W_{r-s-1})\cup B_0$ be such
that $\deg F_i=d_i$ for $1\le i\le s$. We claim that
\begin{itemize}
  \item $\bfs F_s$ defines an ideal-theoretic complete intersection.
  \item $V(\bfs F_s)$ is irreducible.
\end{itemize}
Observe that, if $\bfs F_s\notin\pi(W_{r-s-1})\cup B_0$, then $\bfs
F_s\notin\pi(W_{r-s})\cup B_0$. Therefore, the first assertion is a
consequence of Theorem \ref{th: dimension F_s not red reg seq}. In
particular, $ V(\bfs F_s)$ is of pure dimension $r-s$, and thus the
projective closure of $ V(\bfs F_s)$ is of pure dimension $r-s$. On
the other hand, such a projective closure is contained in the
irreducible variety $V(\bfs F_s^h)$, of dimension $r-s$. We conclude
that $V(\bfs F_s^h)$ is the projective closure of $V(\bfs F_s)$,
which implies that $V(\bfs F_s)$ is irreducible.

We observe that the proof of Lemma \ref{lemma: dimension B_0} also
shows that, for $d_i\ge 1$ for $1\le i\le s$,  $B_0$ is a subvariety
of $\mathcal{F}_{\bfs d_s}(\overline{K})$ of codimension at least
$r-s+1$ (see Remark \ref{rem: dimension B_0 for d_i=1}). As a
consequence, any $\bfs F_s\in B_2$ is either contained in $\pi
(W_{r-s-1})\cup B_0$, which is a subvariety of $\mathcal{F}_{\bfs
d_s}(\overline{K})$ of codimension at least $r-s+1$, or is defined
by polynomials $F_1,\ldots,F_s$ with $\deg F_i<d_i$ for some $i$.
%Arguing in this way for $1\le j\le s$, we conclude that the set of
%$\bfs F_s:=(F_1,\ldots,F_s)$ with $\deg F_i=d$ for $1\le i\le s$
%such that $F_1,\ldots,F_j$ is a reduced regular sequence is
%contained in a subvariety of $\mathcal{F}_{\bfs d_s}(\overline{K})$
%of codimension at least $r+1-j$.
It follows that the set of $\bfs F_s$ such that $ V(\bfs F_s)$ is
not irreducible or an ideal-theoretic complete intersection is
contained in a subvariety of $\mathcal{F}_{\bfs d_s}(\overline{K})$
of codimension at least $r-s+1$.
\end{proof}

\begin{remark}\label{rem: expression B_2}
By the proof of Proposition \ref{th: dimension F_s not abs irred} we
have that $B_2$ is contained in the union
$$
B_2\subset  \bigcup_{i=1}^s L_i  \cup \pi(W_{r-s-1})\cup B_0,
$$
where $L_i:=\mathcal{F}_{d_1}(\overline{K})\times\cdots
\times\mathcal{F}_{d_{i-1}}(\overline{K})\times
\mathcal{F}_{d_i-1}(\overline{K})\times\mathcal{F}_{d_{i+1}}(\overline{K})
\times\cdots \times\mathcal{F}_{d_s}(\overline{K})$  for $1\le i\le
s$ and $B_0$ is defined in Lemma \ref{lemma: dimension B_0}
\end{remark}

One may wonder about the optimality of the bound on the codimension
of $B_2$ of Theorem \ref{th: dimension F_s not abs irred}. With
respect to this point, we have the following remark.
\begin{remark}
For $d_1=\cdots=d_s=1$, any $\bfs F_s\in
\mathcal{F}_1^s(\overline{K})$ is defined by linear polynomials
$F_i:=a_{i,0}+a_{i,1}X_1+\cdots+a_{i,r}X_r$ $(1\le i\le s)$.
Further, $\bfs F_s$ belongs to the set $B_2$ of \eqref{def: B} if
and only if the matrix $(a_{i,j})_{1\le i\le s,1\le j\le r}$ has
rank at most $s-1$. According to \cite[Proposition 1.1]{BrVe88}, the
set of matrices in $\overline{K}{}^{s\times r}$ of rank at most
$s-1$ forms a subvariety of $\overline{K}{}^{s\times r}$ of
dimension
$$(s-1)(r+1)=s\,r-r+s-1.$$
In other words, the set of such matrices is a variety of codimension
$r-s+1$. It follows that set $B_2$ of \eqref{def: B} for
$d_1=\cdots=d_s=1$ has codimension exactly $r-s+1$ in
$\mathcal{F}_1^s(\overline{K})$.
\end{remark}

%
%----------------------------------------------------------------------------
%----------------------------------------------------------------------------
%----------------------------------------------------------------------------
%----------------------------------------------------------------------------
%----------------------------------------------------------------------------
%----------------------------------------------------------------------------
%----------------------------------------------------------------------------
%----------------------------------------------------------------------------
%
\section{The number of defective systems over a finite field}
\label{section: number defective systems}
In this section we deal with $K=\fq$ and aim to estimate the number
of defective systems defined over $\fq$. More precisely, for $K=\fq$
and $B_1$ and $B_2$ defined as in \eqref{def: B_1} and \eqref{def:
B}, we aim to estimate the numbers $|B_1(\fq)|$ and $|B_2(\fq)|$ of
systems $\bfs F_s\in \mathcal{F}_{\bfs d_s}$ such that $V(\bfs F_s)$
either is not an ideal-theoretic complete intersection or is not
absolutely irreducible.
%
%----------------------------------------------------------------------------
%----------------------------------------------------------------------------
%----------------------------------------------------------------------------
%----------------------------------------------------------------------------
%
\subsection{Systems not defining an ideal-theoretic complete intersection}
\label{subsec: number systems not H}
Recall that Theorem \ref{th: dimension F_s not red reg seq} asserts
that $B_1$ has codimension at least $r-s+2$. Further, Remark
\ref{rem: expression B_2} shows that
$$
B_1\subset\bigcup_{i=1}^s L_i  \cup \pi(W_{r-s})\cup B_0,
$$
where $L_i:=\mathcal{F}_{d_1}(\cfq)\times\cdots\times
\mathcal{F}_{d_{i-1}}(\cfq)\times
\mathcal{F}_{d_i-1}(\cfq)\times\mathcal{F}_{d_{i+1}}(\cfq)
\times\cdots \times\mathcal{F}_{d_s}(\cfq)$ for $1\le i\le s$ and
$B_0$ is defined in Lemma \ref{lemma: dimension B_0}. Observe that
\begin{equation}\label{eq: bound points L_i}
|L_i(\fq)|=q^{\dim \mathcal{F}_{\bfs d_s}-\binom{d_i+r-1}{r-1}}\leq
q^{\dim \mathcal{F}_{\bfs d_s}-(r-s+2)}
\end{equation}
for $1\le i\le s$, where $\dim \mathcal{F}_{\bfs d_s}$ denotes the
dimension of $\mathcal{F}_{\bfs d_s}$ as $\fq$-vector space. Hence,
\begin{align}
|B_1(\fq)|&\le |\pi(W_{r-s})(\fq)|+\sum_{i=1}^s
|L_i(\fq)|+|B_0(\fq)|\notag\\ &\le |\pi(W_{r-s})(\fq)|+s\, q^{\dim
\mathcal{F}_{\bfs d_s}-(r-s+2)}+|B_0(\fq)|.\label{eq: expression B_1
bound degree}
\end{align}

We start with a bound for $|B_0(\fq)|$. In the proof of Lemma
\ref{lemma: dimension B_0} it is shown that
$$B_0\subset B_{r-s}^{\mathrm{in}}:=\{\bfs F_s\in\mathcal{F}_{\bfs d_s}(\cfq)
:\dim V(\bfs F_s^{\mathrm{in}})\ge r-s\}.$$
Therefore, it suffices to bound the number of $\fq$-rational points
of $B_{r-s}^{\mathrm{in}}$. For this purpose, we show that
$B_{r-s}^{\mathrm{in}}$ can be described by equations of low degree.

Let $\bfs\Phi:=(\phi_{i,j}:1\le i\le r,1\le j\le r-s)$ be an
$r\times (r-s)$ matrix of indeterminates over $\cfq$ and denote
$$
\mathcal{L}_j:=\sum_{i=1}^r\phi_{i,j}X_i\quad(1\le j\le r-s).
$$
We shall see that the condition $\dim V(\bfs F_s^{\mathrm{in}})\ge
r-s$ holds if and only if
\begin{equation}\label{eq: system for resultant B_0}
\{\bfs x\in\Pp^{r-1}(\overline{\fq(\bfs\Phi)}): \bfs
F_s^{\mathrm{in}}(\bfs
x)=\mathcal{L}_1=\cdots=\mathcal{L}_{r-s}=0\}\not=\emptyset.
\end{equation}
For this aim, we have the following result.
%By \cite[Exercise I.3.28]{Kollar96}, the condition $\dim V(\bfs
%F_s^{\mathrm{in}})\ge r-s$ holds if and only if
%
\begin{lemma}\label{lemma: condition Kollar}
Let $K$ be a field and $d\in \mathbb{N}$. Let $V:=V(F_1, \dots,
F_m)\subset \mathbb{P}^n(\overline{K})$ be a nonempty projective
$K$--variety. Let $\bfs \Xi:=(\Xi_{i,j}:1\le i\le d,0\le j\le n)$ be
a $d\times n$ matrix of indeterminates over $K$ and denote $
\Xi_j:=\mbox{$\sum_{i=0}^n$}\Xi_{i,j}X_i$ for $1\le j\le d$. Then
$\dim V \geq d$ if and only if
\begin{align*}
\{\bfs x \in \mathbb{P}^n(\overline{K(\bfs \Xi)}): F_1(\bfs x)=0,
\dots, F_m(\bfs x)=0, \Xi_1(\bfs x)=0, \dots, \Xi_d(\bfs x)=0\}\neq
\emptyset.
\end{align*}
\end{lemma}
\begin{proof} Suppose that $\dim V \geq d$ and let
$V^*:=V_{\mathbb{P}^n(\overline{K(\bfs \Xi)})}(F_1, \dots, F_m)$.
Then $\dim V^*\geq d$. As $\Xi_1, \dots, \Xi_d$ are
linearly-independent linear forms, we have $\dim
V_{\mathbb{P}^n(\overline{K(\bfs \Xi)})}(\Xi_1, \dots, \Xi_d) =n-d$.
Thus $\dim V^* + \dim V_{\mathbb{P}^n(\overline{K(\bfs
\Xi)})}(\Xi_1, \dots, \Xi_d)\geq n$, which implies
$$
V^*\cap V_{\mathbb{P}^n(\overline{K(\bfs \Xi)})}(\Xi_1, \dots,
\Xi_d)\neq \emptyset.
$$
Conversely, if this intersection is nonempty, then by \cite[Chapter
X, \S.5, Lemma II]{HoPe68b} it follows that $\dim V^*\geq d$, which
implies that $\dim V \geq d$.
\end{proof}

Applying Lemma \ref{lemma: condition Kollar} to our situation we
deduce that the condition $\dim V(\bfs F_s^{\mathrm{in}})\ge r-s$
holds if and only if \eqref{eq: system for resultant B_0} holds. Now
we can proceed to prove that $B_{r-s}^{\mathrm{in}}$ can be
described by equations of low degree.
\begin{lemma}\label{lemma: equations for B_(r-s)^in}
Let $\delta:=d_1\cdots d_s$. If $d_i\ge 2$ for $1\le i\le s$, then
there exist multihomogeneous polynomials $\Phi_1,\ldots,\Phi_L\in
\fq[\mathrm{coeffs}(\bfs F_s)]$, of degree
$\frac{\delta}{d_i}$
in the variables $\mathrm{coeffs}(F_i)$ for $1\le i\le s$, such that
$V(\Phi_1,\ldots,\Phi_L)=B_{r-s}^{\mathrm{in}}$.
\end{lemma}
\begin{proof}
Condition \eqref{eq: system for resultant B_0} can be expressed by
means of multivariate resultants. The multivariate resultant of
formal homogeneous polynomials ${\sf F}_1 \klk {\sf
F}_r\in\K[X_1,\ldots,X_r]$ of degrees $d_1\klk d_r$ over a field
$\K$ is the unique irreducible multihomogeneous polynomial
$\mathrm{Res}\in\K[\mathrm{coeffs}({\sf F}_1)\klk
\mathrm{coeffs}({\sf F}_r)]$ with the following properties (see
\cite[Chapter 3, Theorem 2.3]{CoLiOS98}):
\begin{itemize}
  \item if $G_1,\ldots,G_r\in
\K[X_1,\ldots,X_r]$ are homogeneous polynomials of degrees $d_1\klk
d_r$, then $\mathrm{Res}(G_1,\ldots,G_r)=0$ if and only if
\linebreak $V(G_1,\ldots,G_r)\subset \Pp^r(\overline{\K})$ is
nonempty.
\item $\mathrm{Res}(X_1^{d_1},\ldots,X_r^{d_r})=1$.
\end{itemize}
Further, $\mathrm{Res}$ has degree $d_1\cdots d_{i-1}d_{i+1}\cdots
d_r$ in the coefficients of ${\sf F}_i$; see \cite[Chapter 3,
Theorem 3.1]{CoLiOS98}. As a consequence, the multivariate resultant
of $\bfs F_s^{\mathrm{in}},\mathcal{L}_1,\ldots,\mathcal{L}_{r-s}$
is a nonzero multihomogeneous polynomial $P_0\in
\fq[\bfs\Phi,\mathrm{coeffs}(\bfs F_s)]$ of degree
\begin{itemize}
\item $\frac{\delta}{d_i}$ in the coefficients of each $F_i$,
\item $\delta$ in the coefficients of each $\mathcal{L}_i$.
\end{itemize}
Write $P_0$ as an element of the polynomial ring
$\fq[\mathrm{coeffs}(\bfs F_s)][\bfs\Phi]$. Then the coefficients of
$P_0$ in $\fq[\mathrm{coeffs}(\bfs F_s)]$ determine a system of
equations for the set of system satisfying $\dim V(\bfs
F_s^{\mathrm{in}})>r-1-s$. According to our previous remarks, each
of these polynomials is multihomogeneous of degree
$\frac{\delta}{d_i}$
in the variables $\mathrm{coeffs}(F_i)$ for $1\le i\le s$.
\end{proof}

Next we show that there are varieties containing
$B_{r-s}^{\mathrm{in}}$ which can be described by ``few'' equations
of ``low'' degree.
\begin{lemma}\label{lemma: few eqs for B_(r-s)^in}
With hypothesis as in Lemma \ref{lemma: equations for B_(r-s)^in},
for $1\le k\le r-s+2$ there exist multihomogeneous polynomials
$\Phi^1,\ldots,\Phi^k\in \cfq[\mathrm{coeffs}(\bfs F_s)]$ of degree
$\frac{\delta}{d_i}$ in the variables $\mathrm{coeffs}(F_i)$ for
$1\le i\le s$, such that $V(\Phi^1,\ldots,\Phi^k)$ is of pure
codimension $k$ and contains $B_{r-s}^{\mathrm{in}}$. \end{lemma}
\begin{proof}
Let $\Phi^1$ be a nontrivial linear $\cfq$-combination of the
polynomials \linebreak $\Phi_1,\ldots,\Phi_L$ of the statement of
Lemma \ref{lemma: equations for B_(r-s)^in}. It is clear that
$\Phi^1$ satisfies the conditions of the lemma for $k=1$.

Now assume inductively that we have $\Phi^1,\ldots,\Phi^{j-1}$
satisfying the conditions of the lemma for $k=j-1$. Write
$\mathcal{C}_1,\ldots,\mathcal{C}_t$ for the irreducible
$\cfq$-components of $V(\Phi^1,\ldots,\Phi^{j-1})$. As
$V(\Phi^1,\ldots,\Phi^{j-1})\subset\mathcal{F}_{\bfs d_s}(\cfq)$ is
of pure codimension $j-1$, and $B_{r-s}^{\mathrm{in}}$ is of
codimension at least $r-s+2>j-1$, no $\mathcal{C}_i$ is contained in
$B_{r-s}^{\mathrm{in}}$. Therefore, for each $i$ there exist $\bfs
x^i\in \mathcal{C}_i\setminus B_{r-s}^{\mathrm{in}}$. As
$(\Phi_1(\bfs x^i),\ldots,\Phi_L(\bfs x^i))$ is nonzero for $1\le
i\le t$, there exists a linear $\cfq$-combination $\Phi^j$ of
$\Phi_1,\ldots,\Phi_L$ which does not vanish at any $\bfs x^i$. This
implies that $\{\Phi^j=0\}$ has a nontrivial intersection with all
the components $\mathcal{C}_i$, and hence $V(\Phi^1,\ldots,\Phi^j)=
V(\Phi^1,\ldots,\Phi^{j-1})\cap\{\Phi^j=0\}$ has pure codimension
$j$. Further, by construction it is clear that the second condition
in the lemma holds, finishing thus the proof of the lemma.
\end{proof}

Finally, we can establish our bound on $|B_0(\fq)|$.
\begin{lemma}\label{lemma: bound nmb points B_0}
Let $\delta:=d_1\cdots d_s$ and $1\le k\le r-s+2$. If $d_i\ge 2$ for
$1\le i\le s$, then
$$|B_0(\fq)|\le \Bigg(\sum_{i=1}^s\frac{\delta}{d_i}\Bigg)^k
q^{\dim \mathcal{F}_{\bfs d_s}-k}.$$
\end{lemma}
\begin{proof}
By the B\'ezout inequality \eqref{eq: Bezout}, we have
$$\deg V(\Phi^1,\ldots,\Phi^k)\le
\Bigg(\sum_{i=1}^s\frac{\delta}{d_i}\Bigg)^k,$$
and \eqref{eq: upper bound -- affine gral} implies
$$|B_0(\fq)|\le |B_{r-s}^{\mathrm{in}}(\fq)|\le|V(\Phi^1,\ldots,\Phi^k)(\fq)|\le
\Bigg(\sum_{i=1}^s\frac{\delta}{d_i}\Bigg)^kq^{\dim
\mathcal{F}_{\bfs d_s}-k}.$$
This finishes the proof of the lemma
\end{proof}

Finally, it remains to bound $|\pi(W_{r-s})(\fq)|$. For this
purpose, we show that $\pi(W_{r-s})$ can be expressed with equations
of low degree. In Section \ref{subsec: dimension systems not H} we
prove that
$$W_{r-s}:=\{(\bfs F_s,\bfs x)\in W:\dim\pi^{-1}(\bfs F_s)\ge r-s\},$$
is Zariski closed. According to \eqref{eq: upper bound W_(r-s)}, we
have
$$\dim \mathcal{F}_{\bfs d_s}(\cfq)-(r-s+2)\ge \dim \pi(W_{r-s}).$$
%
%As a consequence, we have
%%
%\begin{equation}\label{eq: bound degree pi(W_(r-s-1)) bis}
%|\pi(W_{r-s-1})(\fq)|\le \deg(\pi(W_{r-s-1}))\,q^{ \dim
%\mathcal{F}_{\bfs d_s}(\cfq) -(r-s+1)}.
%\end{equation}
%
Let $\bfs\Theta_1:=(\theta_i:1\le i\le N)$ be a vector $N$ of new
indeterminates over $\cfq$ and denote
\begin{align*}
\Delta_{\bfs\Theta_1}:=\sum_{i=1}^N\theta_i\Delta_i,
\end{align*}
where $\Delta_1,\ldots,\Delta_N$ are the maximal minors of the
Jacobian matrix of $\bfs F_s^h$ with respect to $X_0,\ldots,X_r$.
Let ${\sf K}:=\cfq(\bfs\Theta_1)$ and let $\overline{{\sf K}}$ be an
algebraic closure of ${\sf K}$. %and let
%$\pi_{\bfs\Theta_1}:\mathcal{F}_{\bfs d_s}(\overline{{\sf K}})\times
%\Pp^r(\overline{{\sf K}})\to \mathcal{F}_{\bfs d_s}(\overline{{\sf
%K}})$ be the projection on the first coordinate.
First we show that the fibers in $W_{r-s}$ can be described using a
single linear combination of $\Delta_1,\ldots,\Delta_N$ (over a
transcendent field extension). For this purpose, we start with the
following technical lemma.
\begin{lemma}\label{lemma: condition linear comb pol}
Let $\K$ be an infinite field. Let $V\subset
\mathbb{P}^r(\overline{\mathbb{K}})$ be a projective variety such
that $\dim Z \geq d \geq 1$ for every irreducible component $Z$ of
$V$, and $\Delta_1, \ldots, \Delta_N$ homogeneous polynomials of the
same degree in $\K[X_0, \ldots, X_r]$ such that  $\dim V \cap
V_{\mathbb{P}^r(\overline{\K})}(\Delta_1, \ldots, \Delta_N) < d$.
Then, for a generic choice of $\bfs \lambda:=(\lambda_1, \ldots,
\lambda_N) \in \K^N$, the hypersurface of
$\mathbb{P}^r(\overline{\K})$ defined by $\lambda_1\Delta_1 + \cdots
+ \lambda_N\Delta_N$ intersects $V$ properly.
\end{lemma}
\begin{proof}
Write $\mathcal{S}:=V \cap V_{\mathbb{P}^r(\overline{\K})}(\Delta_1,
\ldots, \Delta_N)$ and denote by $Z_1, \ldots, Z_m$ the irreducible
components of $V$. As $\dim \mathcal{S} < \dim Z_i$, no component
$Z_i$ is contained in $\mathcal{S}$, and thus there exists a point
$\bfs x_i\in Z_i\setminus \mathcal{S}$. In particular,
$(\Delta_1(\bfs x_i), \ldots, \Delta_N(\bfs x_i))$ is a nonzero
vector for $1\le i \le m$. It follows that, for a generic choice of
$\bfs\lambda\in \K^N$, the polynomial $\Delta_{\bfs
\lambda}:=\lambda_1\Delta_1 + \cdots + \lambda_N\Delta_N$ satisfies
$\Delta_{\bfs \lambda}(\bfs x_i)\neq 0$ for $1\le i \le m$. Thus no
component $Z_i$ is contained in the hypersurface
$V_{\mathbb{P}^r(\overline{\K})}(\Delta_{\bfs\lambda})$. Further, as
$\dim Z_i + \dim
V_{\mathbb{P}^r(\overline{\K})}(\Delta_{\bfs\lambda})\geq d +
r-1\geq r$, we have $Z_i\cap
V_{\mathbb{P}^r(\overline{\K})}(\Delta_{\bfs\lambda})\neq \emptyset$
for $1\le i \le m$ (see, e.g., \cite[Chapter 5, Corollary
3.10]{Kunz85}). This completes the proof of the lemma.
\end{proof}

We need a second technical lemma.
\begin{lemma}\label{lemma: zero_divisor}
Let $\K$ be an infinite field, $X_0,\ldots,X_r$ indeterminates over
$\K$, and $I\subseteq \K[X_0, \ldots, X_r]$ an ideal. Let $\bfs
U:=(U_1, \ldots, U_m)$ be a new set of indeterminates over $\K$.
Suppose that a polynomial $F\in \K(\bfs U)[X_0, \ldots, X_r]$ is a
zero divisor modulo the extended ideal $I\cdot \K(\bfs U)[X_0,
\ldots, X_r]$. Then, for a generic specialization $\bfs U \mapsto
\bfs \lambda\in \K^m$, the specialized polynomial $F(\bfs \lambda)$
is a zero divisor modulo $I$.
\end{lemma}
\begin{proof}
Let $F_1, \ldots, F_s\in \K[X_0, \ldots, X_r]$ be a basis of $I$. By
hypothesis, there exist polynomials $G,H_1, \ldots, H_s\in\K(\bfs
U)[X_0, \ldots, X_r]$ such that
\begin{align}\label{eq: zero_divisor_on_K}
G\notin I\cdot \K(\bfs U)[X_0, \ldots, X_r] \quad \textrm{ and }
\quad FG = H_1F_1+\ldots +H_sF_s.
\end{align}
%
%We claim that these polynomial relations hold by replacing
%$\bfs\Theta_1$ for a generic choice of  $\bfs\lambda \in\cfq{\!}^N$.
Fix a monomial order of $\K[X_0, \ldots, X_r]$ and consider the
reduced Gr\"{o}bner basis $Q_1, \ldots, Q_n$  of $I$ with respect to
this monomial ordering. Since $\K(\bfs U)$ is a field extension of
$\K$, the polynomials $Q_1, \ldots, Q_n$ also form a Gr\"{o}bner
basis of the extended ideal $I\cdot \K(\bfs U)[X_0, \ldots, X_r]$.
Let $G_{red}\in \K(\bfs U)[X_0, \ldots, X_r]$ be the normal form of
$G$ with respect to $Q_1, \ldots, Q_n$. Thus, there exist $K_1,
\ldots, K_n\in \K(\bfs U)[X_0, \ldots, X_r]$ with
$$
G=K_1Q_1+ \ldots +K_nQ_n + G_{red}.
$$
As $G\notin I\cdot \K(\bfs U)[X_0, \ldots, X_r]$ we have
$G_{red}\neq 0$. Choose $\bfs\lambda \in \K^m$ such that
substituting $\bfs\lambda$ for $\bfs U$ does not annihilate any
denominator appearing in $G, K_1, \ldots, K_n$ and such that neither
$G(\bfs\lambda, \bfs X)$ nor $G_{red}(\bfs\lambda, \bfs X)$ are the
zero polynomial. For such a $\bfs\lambda$ we have
$$
G(\bfs\lambda, \bfs X)=K_1(\bfs\lambda, \bfs X)Q_1(\bfs X)+ \ldots
+K_m(\bfs\lambda, \bfs X)Q_n(\bfs X) + G_{red}(\bfs\lambda, \bfs X)
$$
Note that a monomial of any non zero term of $G_{red}(\bfs\lambda,
\bfs X)$ is also a monomial of $G_{red}$. Since $G_{red}$ is reduced
with respect to $Q_1, \ldots, Q_n$  we deduce that
$G_{red}(\bfs\lambda, \bfs X)$ is also reduced with respect to $Q_1,
\ldots, Q_n$. Thus $G_{red}(\bfs\lambda, \bfs X)$ is the normal form
of $G$ with respect to the  Gr\"{o}bner basis $Q_1, \ldots, Q_n$ of
$I$. Since $G_{red}(\bfs\lambda, \bfs X)\neq 0$, we deduce that
$G(\bfs\lambda, \bfs X)\notin I$. If we choose $\bfs\lambda$ so
that, in addition, no denominator in $H_1, \ldots, H_s$ and $F$
vanishes by substituting $\bfs\lambda$ for $\bfs U$, then the
following relations hold:
$$
G(\bfs\lambda,\! \bfs X)\notin I \ \textrm{and} \ F(\bfs\lambda,
\!\bfs X)\cdot G(\bfs\lambda,\! \bfs X) = H_1(\bfs\lambda,\! \bfs
X)F_1(\bfs X)+\cdots + H_s(\bfs\lambda,\! \bfs X)F_s(\bfs X).
$$
This shows that $ F(\bfs\lambda, \bfs X)$ is a zero divisor modulo
$I$, which proves the lemma.
\end{proof}

Now we can proceed to show that the fibers in $W_{r-s}$ can be
described using a single linear combination of
$\Delta_1,\ldots,\Delta_N$.
\begin{proposition}\label{prop: condition linear comb minors}
For any $\bfs F_s\in\mathcal{F}_{\bfs d_s}(\cfq)$ we have
\begin{align*}
\dim\pi^{-1}(\bfs F_s)\ge\! r-s\Leftrightarrow \dim\{\bfs
x\in\Pp^r(\overline{{\sf K}}): \bfs F_s^h(\bfs
x)=\!\Delta_{\bfs\Theta_1}(\bfs F_s^h,\bfs x)=0\}\!\ge
r-s.\end{align*}
\end{proposition}
\begin{proof}
Suppose that $V(\bfs F^h_s, \Delta_1, \ldots, \Delta_N)\subset\Pp^r$
has an irreducible component $\mathcal{C}$ of dimension at least
$r-s$. Let $G_1, \ldots, G_h$ be homogeneous polynomials in
$\cfq[X_0, \ldots, X_r]$ defining $\mathcal{C}$. By, e.g.,
\cite[Chapter II, Proposition 4.4, c)]{Kunz85}, the subvariety
$\mathcal{C}^*$ of $\Pp^r(\overline{{\sf K}})$ defined by $G_1,
\ldots, G_h$ is of pure dimension equal to $\dim \mathcal{C}$.
Further, we claim that $\mathcal{C}^*\subseteq
V_{\Pp^r(\overline{{\sf K}})}(\bfs F^h_s, \Delta_{\bfs\Theta_1})$.
Indeed, since each $F^h_i$ vanishes identically on $\mathcal{C}$, by
the Nullstellensatz if follows that a power of  $F^h_i$ belongs to
the ideal of $\cfq[X_0, \ldots, X_r]$ generated by $G_1, \ldots,
G_h$, and thus $F^h_i$ vanishes identically on $\mathcal{C}^*$. By
the same argument, all the $\Delta_i$, and therefore
$\Delta_{\bfs\Theta_1}$, vanish identically on $\mathcal{C}^*$. This
proves our claim. We conclude that $\dim V_{\Pp^r(\overline{{\sf
K}})}(\bfs F^h_s, \Delta_{\bfs\Theta_1}) \geq \dim \mathcal{C}^*\geq
r-s$.

Now assume that $\dim V(\bfs F^h_s, \Delta_1, \ldots, \Delta_N)\le
r-s-1$. We claim that this implies $\dim V(\bfs F^h_s)=r-s$. Suppose
on the contrary that $V(\bfs F^h_s)$ has a component $\mathcal{C}$
with $\dim \mathcal{C} > r-s$. Then the tangent space of
$\mathcal{C}$ at any point $\bfs x\in \mathcal{C}$ has dimension
greater than $r-s$. It follows that the Jacobian matrix of $F^h_1,
\ldots F^h_s$ with respect to $X_0, \ldots, X_r$ has rank less than
$s$ and therefore all the maximal minors $\Delta_1, \ldots,
\Delta_N$ of this matrix vanish at $\bfs x$. We conclude that
$\mathcal{C}\subseteq V(\bfs F^h_s, \Delta_1, \ldots, \Delta_N)$,
and thus that $\dim V(\bfs F^h_s, \Delta_1, \ldots, \Delta_N)> r-s$,
contradicting our assumption. This proves our claim.

If $\bfs\lambda:=(\lambda_1, \ldots, \lambda_N)\in \cfq{\!}^N$, then
we write $\Delta_{\bfs\lambda}:=\lambda_1\Delta_1 + \ldots +
\lambda_N\Delta_N$ for the polynomial in $\cfq[X_0, \ldots, X_r]$
obtained by substituting $\lambda_i$ for $\theta_i$ in
$\Delta_{\bfs\Theta_1}$ for $1\le i \le N$. We claim that for a
generic choice $\bfs\lambda\in\cfq{\!}^N$, the variety $V(\bfs
F^h_s, \Delta_{\bfs\lambda})$ is of pure dimension $r-s-1$. Indeed,
write $\mathcal{S}:=V(\bfs F^h_s, \Delta_1, \ldots, \Delta_N)$ and
denote by $Z_1, \ldots, Z_m$ the irreducible components of $V(\bfs
F^h_s)$. Since $\dim V(\bfs F^h_s)=r-s$,  we have $\dim Z_i=r-s$ for
$1\le i \le m$.  As $\dim \mathcal{S} < \dim Z_i$ for each $i$, the
claim follows by Lemma \ref{lemma: condition linear comb pol}.

Next we show that $V_{\mathbb{P}^r(\overline{{\sf K}})}(\bfs F^h_s,
\Delta_{\bfs\Theta_1})$ is of pure dimension $r-s-1$. Arguing by
contradiction, assume that $\dim V_{\mathbb{P}^r(\overline{\sf
K})}(\bfs F^h_s, \Delta_{\bfs\Theta_1})
> r-s-1$. As $V(\bfs F^h_s)$ is of pure dimension $r-s$, then so is
$V_{\mathbb{P}^r(\overline{\sf K})}(\bfs F^h_s)$, and our assumption
implies that $\Delta_{\bfs\Theta_1}$ is a zero divisor modulo the
ideal $(F^h_1, \ldots, F^h_s){\sf K}[X_0, \ldots, X_r]$. By Lemma
\ref{lemma: zero_divisor}, for a generic specialization $\bfs
\Theta_1 \mapsto \bfs \lambda \in \cfq{\!}^N$, the specialized
polynomial $\Delta_{\bfs\lambda}$ is a zero divisor modulo $(F^h_1,
\ldots, F^h_s)$, which contradicts the fact that $V(\bfs F^h_s,
\Delta_{\bfs\lambda})$ is of pure dimension $r-s-1$ for generic
$\bfs\lambda$. Thus $V_{\mathbb{P}^r(\overline{\sf K})}(\bfs F^h_s,
\Delta_{\bfs\Theta_1})$ is of pure dimension $r-s-1$.
\end{proof}

Next, let $\bfs\Xi:=(\xi_{i,j}:0\le i\le r,1\le j\le r-s)$ be an
$(r+1)\times (r-s)$ matrix of indeterminates over $\overline{{\sf
K}}$ and denote
\begin{align*}
\Xi_j:=\sum_{i=0}^r\xi_{i,j}X_i\quad(1\le j\le r-s).
\end{align*}
Lemma \ref{lemma: condition Kollar} proves that the condition in
Proposition \ref{prop: condition linear comb minors} holds if and
only if
\begin{equation}\label{eq: system for resultant B_1}
\{\bfs x\in\Pp^r(\overline{{\sf K}(\bfs\Xi)}): \bfs F_s^h(\bfs
x)=\Delta_{\bfs\Theta_1}(\bfs F_s^h,\bfs
x)=\Xi_1=\cdots=\Xi_{r-s}=0\}
\end{equation}
is nonempty as a subset of $\Pp^r(\overline{{\sf K}(\bfs\Xi)})$. Now
we can show that $\pi(W_{r-s})$ can be described by equations of low
degree.
\begin{lemma}
Let $\delta:=d_1\cdots d_s$ and $\sigma:=d_1+\cdots+d_s-s$. There
exist multihomogeneous polynomials $\Phi_1,\ldots,\Phi_M\in
\fq[\mathrm{coeffs}(\bfs F_s)]$ of degree
$$C_1(d_i,r,s):=\delta\,\sigma\Big(1+\frac{1}{d_i}\Big)$$
in the variables $\mathrm{coeffs}(F_i)$ for $1\le i\le s$, such that
$V(\Phi_1,\ldots,\Phi_M)=\pi(W_{r-s})$.
\end{lemma}
\begin{proof}
The condition that the variety of \eqref{eq: system for resultant
B_1} is nonempty can be expressed by means of multivariate
resultants. Taking into account $\Delta_{\bfs\Theta_1}(\bfs
F_s^h,-)$ has degree $\sigma$ in the variables $\bfs X$ and $1$ in
the coefficients of each $F_i$, we conclude that the multivariate
resultant of $\bfs F_s^h,\Delta_{\bfs\Theta_1}(\bfs F_s^h,-)$,
$\Xi_1,\ldots,\Xi_{r-s}$ is a nonzero multihomogeneous polynomial
$P\in \fq[\bfs\Xi,\bfs\Theta_1,\mathrm{coeffs}(\bfs F_s)]$ of degree
\begin{itemize}
\item $d_{\bfs F_s}:=\frac{\delta}{d_i}\,\sigma$ in the coefficients of each $F_i$,
\item $d_{\bfs\Theta}:=\delta$ in the coefficients of $\Delta_{\bfs\Theta_1}(\bfs F_s^h,-)$,
\item $\delta\,\sigma$ in the coefficients of each $\Xi_i$.
\end{itemize}
Write $P$ as an element of the polynomial ring
$\fq[\mathrm{coeffs}(\bfs F_s)][\bfs\Xi,\bfs\Theta_1]$. Then the
coefficients of $P$ in $\fq[\mathrm{coeffs}(\bfs F_s)]$ determine a
system of equations for $\pi(W_{r-s})$. According to our previous
remarks, each of these polynomials is multihomogeneous of degree
$$C_1(d_i,r,s):=d_{\bfs F_s}+d_{\bfs\Theta}=\delta\,\sigma\Big(1+\frac{1}{d_i}\Big)$$
in the variables $\mathrm{coeffs}(F_i)$ for $1\le i\le s$.
\end{proof}

Following the proof of Lemma \ref{lemma: few eqs for B_(r-s)^in}
{\em mutatis mutandis} we obtain the following result.
\begin{lemma}\label{lemma: existence Gamma_1,...,Gamma_(r-s+1) for B_1}
For $1\le j\le r-s+2$ there exist multihomogeneous polynomials
$\Gamma_1,\ldots,\Gamma_j\in\cfq[\mathrm{coeffs}(\bfs F_s)]$, of
degree $C_1(d_i,r,s)$ in the variables $\mathrm{coeffs}(F_i)$ for
$1\le i\le s$, with the following properties:
\begin{itemize}
  \item  $V(\Gamma_1,\ldots,\Gamma_j)\supset \pi(W_{r-s})$.
  \item
  $V(\Gamma_1,\ldots,\Gamma_j)\subset\mathcal{F}_{\bfs d_s}(\cfq)$ is
  of pure codimension $j$.
\end{itemize}
\end{lemma}

Now we are finally able to establish our bound on the number of
systems defined over $\fq$ which do not define an ideal-theoretic
complete intersection.
\begin{theorem}\label{th: number systems not H} Let $\bfs d_s:=
(d_1,\ldots,d_s)$, $\delta:=d_1\cdots d_s$ and $\sigma:=d_1+\cdots+
d_s -s$. If $d_i\ge 2$ for $1\le i\le s$, then $|B_1(\fq)|\leq
\big(2\,s\,\delta\,\sigma \big)^{r-s+2}q^{\dim \mathcal{F}_{\bfs
d_s}-r+s-2}$.
\end{theorem}
\begin{proof}
According to Lemma \ref{lemma: existence Gamma_1,...,Gamma_(r-s+1)
for B_1}, by the B\'ezout inequality \eqref{eq: Bezout}, we have
$$\deg V(\Gamma_1,\ldots,\Gamma_{r-s+2})\le
\Bigg(\sum_{i=1}^sC_1(d_i,r,s)\Bigg)^{r-s+2},$$
and \eqref{eq: upper bound -- affine gral} implies
\begin{align*}
|\pi(W_{r-s})(\fq)|&\le
|V(\Gamma_1,\ldots,\Gamma_{r-s+2})(\fq)|\\&\le
\Bigg(\sum_{i=1}^sC_1(d_i,r,s)\Bigg)^{r-s+2}q^{\dim
\mathcal{F}_{\bfs d_s}-(r-s+2)}.\end{align*}
By \eqref{eq: expression B_1 bound degree}, Lemma \ref{lemma: bound
nmb points B_0} and this bound, we obtain
$$|B_1(\fq)|\leq \Bigg(\bigg(\sum_{i=1}^sC_1(d_i,r,s)\bigg)^{r-s+2}+s
+\bigg(\sum_{i=1}^s\frac{\delta}{d_i}\bigg)^{r-s+2}\Bigg)q^{\dim
\mathcal{F}_{\bfs d_s}-r+s-2}.$$
To deduce the estimate of the theorem, we observe that it suffices
to show that
$$s\le 2\,s\,\delta\,\sigma-\delta\,\sigma\sum_{i=1}^s\Big(1+\frac{1}{d_i}\Big)
-\delta\sum_{i=1}^s\frac{1}{d_i}.$$
As $d_i\ge 2$ for each $i$ and thus $\sigma\ge s$, the theorem
readily follows.
\end{proof}

We can express Theorem \ref{th: number systems not H} in terms of
probabilities.
\begin{corollary}\label{coro: number systems not H}
With notations as in Theorem \ref{th: number systems not H},
considering the uniform probability in $\mathcal{F}_{\bfs d_s}$, the
probability ${\sf P}_1$ of the set of $\bfs F_s$ in
$\mathcal{F}_{\bfs d_s}$ which do not define an ideal-theoretic
complete intersection is bounded in the following way:
$${\sf P}_1\le  \bigg(\frac{2\,s\,\delta\,\sigma}{q}\bigg)^{r-s+2}.$$
\end{corollary}
%
%----------------------------------------------------------------------------
%----------------------------------------------------------------------------
%----------------------------------------------------------------------------
%----------------------------------------------------------------------------
%
\subsection{The number of systems not defining an absolutely irreducible variety}
In Theorem \ref{th: dimension F_s not abs irred} we prove that $B_2$
has codimension at least $r-s+1$. Further, according to Remark
\ref{rem: expression B_2}, we have
$$
B_2\subset  \bigcup_{i=1}^s L_i  \cup \pi(W_{r-s-1})\cup B_0,
$$
where $L_i:=\mathcal{F}_{d_1}(\cfq)\times\cdots\times
\mathcal{F}_{d_{i-1}}(\cfq)\times
\mathcal{F}_{d_i-1}(\cfq)\times\mathcal{F}_{d_{i+1}}(\cfq)
\times\cdots\times \mathcal{F}_{d_s}(\cfq)$ for $1\le i\le s$ and
$B_0$ is defined in Lemma \ref{lemma: dimension B_0}. Hence, by
\eqref{eq: bound points L_i} we have
\begin{align}
|B_2(\fq)|&\le |\pi(W_{r-s-1})(\fq)|+\sum_{i=1}^s
|L_i(\fq)|+|B_0(\fq)|\notag\\ &\le |\pi(W_{r-s-1})(\fq)|+s\, q^{\dim
\mathcal{F}_{\bfs d_s}-(r-s+1)}+|B_0(\fq)|. \label{eq: expression B
bound degree}
\end{align}

Next we bound $|\pi(W_{r-s-1})(\fq)|$. For this purpose, following
the general lines of argumentation of Section \ref{subsec: number
systems not H}, we show that $\pi(W_{r-s-1})$ can be described by
means expressed with equations of low degree. Recall that the
variety
$$W_{r-s-1}:=\{(\bfs F_s,\bfs x)\in W:\dim\pi^{-1}(\bfs F_s)\ge r-s-1\},$$
is Zariski closed. According to \eqref{eq: upper bound W_(r-s+1)},
we have
$$\dim \mathcal{F}_{\bfs d_s}(\cfq)-(r-s+1)\ge \dim \pi(W_{r-s-1}).$$
%
%As a consequence, we have
%%
%\begin{equation}\label{eq: bound degree pi(W_(r-s-1)) bis}
%|\pi(W_{r-s-1})(\fq)|\le \deg(\pi(W_{r-s-1}))\,q^{ \dim
%\mathcal{F}_{\bfs d_s}(\cfq) -(r-s+1)}.
%\end{equation}
%
Let $\bfs\Theta:=(\theta_{i,j}:1\le i\le N,1\le j\le 2)$ be an
$N\times 2$ matrix of new indeterminates over $\cfq$ and denote
\begin{align*}
\Delta_{\bfs\Theta_j}:=\sum_{i=1}^N\theta_{i,j}\Delta_i\quad(1\le
j\le 2),
\end{align*}
where $\Delta_1,\ldots,\Delta_N$ are the maximal minors of the
Jacobian matrix of $\bfs F_s^h$ with respect to $X_0,\ldots,X_r$.
Let ${\sf L}:=\cfq(\bfs\Theta)={\sf K}(\bfs \Theta_2)$ and let
$\overline{{\sf L}}$ be an algebraic closure of ${\sf L}$. %Denote by
%$\pi_{\bfs\Theta}:\mathcal{F}_{\bfs d_s}(\overline{{\sf L}})\times
%\Pp^r(\overline{{\sf L}})\to \mathcal{F}_{\bfs d_s}(\overline{{\sf
%L}})$ the projection on the first coordinate.

\begin{proposition}\label{prop: condition two linear comb minors}
For any $\bfs F_s\in\mathcal{F}_{\bfs d_s}(\cfq)$ we have
\begin{align*}
\dim\pi^{-1}(\bfs F_s)\ge r-s-1\Leftrightarrow\hskip5cm \\\dim\{\bfs
x\in\Pp^r(\overline{{\sf L}}): \bfs F_s^h(\bfs
x)=\Delta_{\bfs\Theta_1}(\bfs F_s^h,\bfs
x)=\Delta_{\bfs\Theta_2}(\bfs F_s^h,\bfs x)=0\}\ge
r-s-1.\end{align*}
\end{proposition}
\begin{proof}
Suppose that $V(\bfs F^h_s, \Delta_1, \ldots, \Delta_N)\subset\Pp^r$
has an irreducible component $\mathcal{C}$ of dimension at least
$r-s-1$. Let $G_1, \ldots, G_t$ be homogeneous polynomials in
$\cfq[X_0, \ldots, X_r]$ defining $\mathcal{C}$. By, e.g.,
\cite[Chapter II, Proposition 4.4, c)]{Kunz85}, the subvariety
$\mathcal{C}^*$ of $\Pp^r(\overline{{\sf K}})$ defined by $G_1,
\ldots, G_t$ is of pure dimension equal to $\dim \mathcal{C}$.
Further, we claim that $\mathcal{C}^*\subseteq
V_{\Pp^r(\overline{{\sf L}})}(\bfs F^h_s, \Delta_{\bfs\Theta_1},
\Delta_{\bfs\Theta_2})$. Indeed, since each $F^h_i$ vanishes
identically on $\mathcal{C}$,  by the Nullstellensatz it follows
that a power of  $F^h_i$ belongs to the ideal of $\cfq[X_0, \ldots,
X_r]$ generated by $G_1, \ldots, G_t$, and thus $F^h_i$ vanishes
identically on $\mathcal{C}^*$. By the same argument, all the
$\Delta_i$, and therefore $\Delta_{\bfs\Theta_1}$ and
$\Delta_{\bfs\Theta_2}$, vanish identically on $\mathcal{C}^*$,
which proves our claim. We conclude that $\dim
V_{\Pp^r(\overline{{\sf L}})}(\bfs F^h_s, \Delta_{\bfs\Theta_1},
\Delta_{\bfs\Theta_2}) \geq \dim \mathcal{C}^*\geq r-s-1$.

Now assume that $\dim V(\bfs F^h_s, \Delta_1, \ldots, \Delta_N)\le
r-s-2$.
%We claim that this implies $\dim V(\bfs F^h_s)=r-s$. Suppose
%on the contrary that $V(\bfs F^h_s)$ has a component $\mathcal{C}$
%with $\dim \mathcal{C} > r-s$. Then the tangent space of
%$\mathcal{C}$ at any point $\bfs x\in \mathcal{C}$ has dimension
%greater than $r-s$. It follows that the Jacobian matrix of $F^h_1,
%\ldots F^h_s$ with respect to $X_0, \ldots, X_r$ has rank less than
%$s$ and therefore all the maximal minors $\Delta_1, \ldots,
%\Delta_N$ of this matrix vanish at $\bfs x$. We conclude that
%$\mathcal{C}\subseteq V(\bfs F^h_s, \Delta_1, \ldots, \Delta_N)$,
%and thus that $\dim V(\bfs F^h_s, \Delta_1, \ldots, \Delta_N)> r-s$,
%contradicting our assumption. This proves our claim.
As in the proof of Proposition \ref{prop: condition linear comb
minors} we conclude that $V_{\mathbb{P}^r(\overline{{\sf K}})}(\bfs
F^h_s, \Delta_{\bfs\Theta_1})$ is of pure dimension $r-s-1$.

%If $\bfs\lambda:=(\lambda_1, \ldots, \lambda_N)\in \cfq{\!}^N$, then
%we write $\Delta_{\bfs\lambda}:=\lambda_1\Delta_1 + \ldots +
%\lambda_N\Delta_N$ for the polynomial in $\cfq[X_0, \ldots, X_r]$
%obtained by substituting $\lambda_i$ for $\theta_{1,j}$ in
%$\Delta_{\bfs\Theta_1}$, or $\theta_{2,j}$ in
%$\Delta_{\bfs\Theta_2}$, for $1\le j \le N$. Now we show that for
%generic choices of $\bfs\lambda$ and $\bfs\lambda^*$ in
%$\cfq{\!}^N$, the variety $V(\bfs F^h_s, \Delta_{\bfs\lambda},
%\Delta_{\bfs\lambda^*})$ is of pure dimension $r-s-2$. Write
%$\mathcal{S}:=V(\bfs F^h_s, \Delta_1, \ldots, \Delta_N)$ and denote
%by $Z_1, \ldots, Z_m$ the irreducible components of $V(\bfs F^h_s)$.
%Since $\dim V(\bfs F^h_s)=r-s$,  we have $\dim Z_i=r-s$ for $1\le i
%\le m$, and thus $\dim \mathcal{S} < \dim Z_i$ for $1\le i \le m$.
%By Lemma \ref{lemma: condition linear comb pol} it follows that, for
%a generic choice of $\bfs\lambda\in \cfq{\!}^N$, the hypersurface
%$V(\Delta_{\bfs\lambda})$ intersects the variety $V(\bfs F^h_s)$
%properly. We deduce that $V(\bfs F^h_s, \Delta_{\bfs\lambda})$ is of
%pure dimension $r-s-1$ (see, e.g., \cite[Chapter 5, Corollary
%3.2]{Kunz85}). Now, denote $W_1, \ldots, W_n$ the irreducible
%components of $V(\bfs F^h_s, \Delta_{\bfs\lambda})$. As $\dim
%\mathcal{S}< \dim W_i$ for $1\le i \le n$, applying Lemma
%\ref{lemma: condition linear comb pol} again we conclude that, for a
%generic choice of $\bfs\lambda^*\in\cfq{\!}^N$, the variety $V(\bfs
%F^h_s, \Delta_{\bfs\lambda}, \Delta_{\bfs\lambda*})$ is of pure
%dimension $r-s-2$.

Next we show that $V_{\mathbb{P}^r(\overline{{\sf L}})}(\bfs F^h_s,
\Delta_{\bfs\Theta_1}, \Delta_{\bfs\Theta_2})$ is of pure dimension
$r-s-2$.  If $\bfs\lambda:=(\lambda_1, \ldots, \lambda_N)\in {\sf
K}^N$, then we write $\Delta_{\bfs\lambda}:=\lambda_1\Delta_1 +
\ldots + \lambda_N\Delta_N$ for the polynomial in ${\sf K}[X_0,
\ldots, X_r]$ obtained by substituting $\lambda_j$ for
$\theta_{j,2}$ in $\Delta_{\bfs\Theta_2}$ for $1\le j \le N$.
Suppose that $\Delta_{\bfs\Theta_2}$ is a zero divisor modulo the
ideal $(\bfs F^h_s, \Delta_{\bfs\Theta_1}){\sf L}[X_0, \ldots,
X_r]$. By Lemma \ref{lemma: zero_divisor},  for a generic
specialization $\bfs \Theta_2 \mapsto \bfs\lambda \in {\sf K}^N$ the
specialized polynomial $\Delta_{\bfs\lambda}$ is a zero divisor
modulo the ideal $(\bfs F^h_s, \Delta_{\bfs\Theta_1})$ of ${\sf
K}[X_0, \ldots, X_r]$. On other hand, set
$\mathcal{S}:=V_{\mathbb{P}^r(\overline{{\sf K}})}(\bfs F^h_s,
\Delta_1, \ldots, \Delta_N)$.
 Then $\mathcal{S}=V_{\mathbb{P}^r(\overline{{\sf K}})}(\bfs F^h_s,
\Delta_{\bfs\Theta_1})\cap V_{\mathbb{P}^r(\overline{{\sf
K}})}(\Delta_1, \ldots, \Delta_N)$ and $\dim \mathcal{S}\le r-s-2$
by assumption. In particular, $\dim \mathcal{S} < \dim Z$ for any
irreducible component $Z$ of $V_{\mathbb{P}^r(\overline{{\sf
K}})}(\bfs F^h_s, \Delta_{\bfs\Theta_1})$. By Lemma \ref{lemma:
condition linear comb pol} it follows that, for a generic
specialization $\bfs \Theta_2 \mapsto \bfs\lambda \in {\sf K}^N$,
the hypersurface of $\mathbb{P}^r(\overline{{\sf K}})$ defined by
the specialized polynomial $\Delta_{\bfs\lambda}$  intersects the
variety $V_{\mathbb{P}^r(\overline{{\sf K}})}(\bfs F^h_s,
\Delta_{\bfs\Theta_1})$ properly. This contradicts the fact that
$\Delta_{\bfs\lambda}$ is a zero divisor modulo $(\bfs F^h_s,
\Delta_{\bfs\Theta_1})$ for generic $\bfs\lambda \in {\sf K}^N$.
Thus, $\Delta_{\bfs\Theta_2}$ is not a zero divisor modulo $(\bfs
F^h_s, \Delta_{\bfs\Theta_1}){\sf L}[X_0, \ldots, X_r]$. Since
$V_{\mathbb{P}^r(\overline{{\sf L}})}(\bfs F^h_s,
\Delta_{\bfs\Theta_1})$ is of pure dimension $r-s-1$ we conclude
that $V_{\Pp^r(\overline{\sf L})}(\bfs F^h_s, \Delta_{\bfs\Theta_1},
\Delta_{\bfs\Theta_2})$ is of pure dimension $r-s-2$, as asserted.
\end{proof}

Next, let $\bfs\Xi:=(\xi_{i,j}:0\le i\le r,1\le j\le r-s-1)$ be an
$(r+1)\times (r-s-1)$ matrix of indeterminates over $\overline{{\sf
K}}$ and denote
\begin{align*}
\Xi_j:=\sum_{i=0}^r\xi_{i,j}X_i\quad(1\le j\le r-s-1).
\end{align*}
According to Lemma \ref{lemma: condition Kollar}, the condition in
Lemma \ref{prop: condition two linear comb minors} holds if and only
if
\begin{equation}\label{eq: system for resultant}
\!\!\{\bfs x\!\in\Pp^r(\overline{{\sf K}(\bfs\Xi)})\!:\! \bfs
F_s^h(\bfs x)=\!\Delta_{\bfs\Theta_1}\!(\bfs F_s^h,\bfs
x)=\!\Delta_{\bfs\Theta_2}(\bfs F_s^h,\bfs
x)=\Xi_1=\cdots=\Xi_{r-s-1}=0\}
\end{equation}
is nonempty as a subset of $\Pp^r(\overline{{\sf K}(\bfs\Xi)})$. Now
we can show that $\pi(W_{r-s-1})$ can be described by equations of
low degree.
\begin{lemma}
Let $\delta:=d_1\cdots d_s$ and $\sigma:=d_1+\cdots+d_s-s$. There
exist multihomogeneous polynomials $\Psi_1,\ldots,\Psi_N\in
\cfq[\mathrm{coeffs}(\bfs F_s)]$ of degree
$$C_2(d_i,r,s):=\delta\,\sigma\bigg(\frac{\sigma}{d_i}+2\bigg)$$
in the variables $\mathrm{coeffs}(F_i)$ for $1\le i\le s$, such that
$V(\Psi_1,\ldots,\Psi_M)=\pi(W_{r-s-1})$.
\end{lemma}
\begin{proof}
We express the condition that the variety of \eqref{eq: system for
resultant} is nonempty by means of multivariate resultants. More
precisely, the variety of \eqref{eq: system for resultant} is
nonempty if and only if the multivariate resultant $P\in
\fq[\bfs\Xi,\bfs\Theta,\mathrm{coeffs}(\bfs F_s)]$ of $\bfs
F_s^h,\Delta_{\bfs\Theta_1}(\bfs F_s^h,-),\Delta_{\bfs\Theta_2}(\bfs
F_s^h,-)$, $\Xi_1,\ldots,\Xi_{r-s-1}$ is nonzero. Taking into
account $\Delta_{\bfs\Theta_1}(\bfs
F_s^h,-),\Delta_{\bfs\Theta_2}(\bfs F_s^h,-)$ have degree $\sigma$
in the variables $\bfs X$ and $1$ in the coefficients of each $F_i$,
we conclude that $P$ is a nonzero multihomogeneous polynomial $P\in
\fq[\bfs\Xi,\bfs\Theta,\mathrm{coeffs}(F_1),\ldots,\mathrm{coeffs}(F_s)]$
of degree
\begin{itemize}
\item $d_{\bfs F_s}:=\frac{\delta}{d_i}\,\sigma^2$ in the coefficients of each $F_i$,
\item $d_{\bfs\Theta}:=\delta\,\sigma$ in the coefficients of each $\Delta_{\bfs\Theta_i}(\bfs F_s^h,-)$,
\item $\delta\,\sigma^2$ in the coefficients of each $\Xi_i$.
\end{itemize}
Write $P$ as an element of the polynomial ring
$\fq[\mathrm{coeffs}(\bfs F_s)][\bfs\Xi,\bfs\Theta]$. Then the
coefficients of $P$ in $\fq[\mathrm{coeffs}(\bfs F_s)]$ determine a
system of equations for $\pi(W_{r-s-1})$. According to our previous
remarks, each of these polynomials is multihomogeneous of degree
$$C_2(d_i,r,s):=d_{\bfs F_s}+2d_{\bfs\Theta}=\delta\,\sigma\bigg(\frac{\sigma}{d_i}+2\bigg)$$
in the variables $\mathrm{coeffs}(F_i)$ for $1\le i\le s$.
\end{proof}

%In order to apply \cite[Lemma 1]{CeGaMa13}, we need to obtain an
%upper bound of the irreducible components of $\pi(W^{'}_{r-s-1})$.
%We observe that  for every irreducible component $V$ of
%$\pi(W^{'}_{r-s-1})$ there exists $U$, irreducible component of
%$W^{'}_{r-s-1}$, such that $\pi(U)=V$. From \cite[Theorem
%11.12]{Harris92} we have that $\dim \pi(U)=\dim U -\mu$ where $\mu$
%is the minimum value of $\mu(p)$ on $U$ with $\mu(p):=\dim_{p}V(
%\bfs F_s^h, \Delta^i, (s+1\leq i\leq r+1))$. Hence,
%$\mathrm{codim}\, \pi(U)=\dim \mathcal{F}_{\bfs d_s}-\dim U +\mu\leq \dim
%\mathcal{F}_{\bfs d_s}+r $. From \cite[Lemma 1]{CeGaMa13}

Arguing {\em mutatis mutandis} as in Lemma \ref{lemma: few eqs for
B_(r-s)^in}, we obtain a variety of the same codimension as
$\pi(W_{r-s-1})$ and low degree.
\begin{lemma}\label{lemma: existence Gamma_1,...,Gamma_(r-s+1)}
For $1\le j\le r-s+1$ there exist multihomogeneous polynomials
$\Gamma_1,\ldots,\Gamma_j\in\cfq[\mathrm{coeffs}(\bfs F_s)]$, of
degree $C_2(d_i,r,s)$ in the variables $\mathrm{coeffs}(F_i)$ for
$1\le i\le s$, with the following properties:
\begin{itemize}
  \item  $V(\Gamma_1,\ldots,\Gamma_j)\supset \pi(W_{r-s-1})$.
  \item
  $V(\Gamma_1,\ldots,\Gamma_j)\subset\mathcal{F}_{\bfs d_s}(\cfq)$ is
  of pure codimension $j$.
\end{itemize}
\end{lemma}
%
%\begin{proof}
%Denote by $\Phi_1,\ldots,\Phi_M$ the coefficients in
%$\cfq[\mathrm{coeffs}(\bfs F_s)]$ of the multivariate resultant $P$
%above. Consider any nontrivial linear $\cfq$-combination $\Gamma_1$
%of $\Phi_1,\ldots,\Phi_M$. It is clear that $\Gamma_1$ satisfies the
%conditions above for $j=1$.
%
%Now assume inductively that we have $\Gamma_1,\ldots,\Gamma_{j-1}$
%satisfying the conditions above for $j-1$. Write $C_1,\ldots,C_t$
%for the irreducible $\cfq$-components of
%$V(\Gamma_1,\ldots,\Gamma_j)$. As
%$V(\Gamma_1,\ldots,\Gamma_j)\subset\mathcal{F}_{\bfs d_s}(\cfq)$ is of pure
%codimension $j-1$, and $\pi(W_{r-s-1})$ is of codimension
%$r-s+1>j-1$, no $C_i$ is contained in $\pi(W_{r-s-1})$. Therefore,
%for each $i$ there exist $\bfs x^i\in C_i\setminus \pi(W_{r-s-1})$.
%As $(\Phi_1(\bfs x^i),\ldots,\Phi_M(\bfs x^i))$ is nonzero for $1\le
%i\le t$, there exists a linear $\cfq$-combination $\Gamma_j$ of
%$\Phi_1,\ldots,\Phi_M$ which does not vanish on any $\bfs x^i$. This
%implies that $\{\Gamma_j=0\}$ has a nontrivial intersection with all
%the components $C_i$, and hence $V(\Gamma_1,\ldots,\Gamma_j)=
%V(\Gamma_1,\ldots,\Gamma_{j-1})\cap\{\Gamma_j=0\}$ has pure
%codimension $j$. Further, by construction it is clear that the
%second condition above holds. This finishes the proof.
%\end{proof}
%
Now we are finally able to establish our bound on the number of
systems defined over $\fq$ not defining an absolutely irreducible
variety.
\begin{theorem}\label{th: number systems not AI}
Let $\bfs d_s:= (d_1,\ldots,d_s)$, $\delta:=d_1\cdots d_s$ and
$\sigma:=d_1+\cdots +d_s -s$. If $d_i\ge 2$ for $1\le i\le s$, then
$|B_2(\fq)|\leq \big(2\,s\,\sigma^2\delta\big)^{r-s+1}q^{\dim
\mathcal{F}_{\bfs d_s}-r+s-1}$.
\end{theorem}
\begin{proof}
According to Lemma \ref{lemma: existence Gamma_1,...,Gamma_(r-s+1)},
by the B\'ezout inequality \eqref{eq: Bezout} we have
$$\deg V(\Gamma_1,\ldots,\Gamma_{r-s+1})\le
\Bigg(\sum_{i=1}^sC_2(d_i,r,s)\Bigg)^{r-s+1},$$
and \eqref{eq: upper bound -- affine gral} implies
\begin{align*}
|\pi(W_{r-s-1})(\fq)|&\le
|V(\Gamma_1,\ldots,\Gamma_{r-s+1})(\fq)|\\&\le
\Bigg(\sum_{i=1}^sC_2(d_i,r,s)\Bigg)^{r-s+1}q^{\dim
\mathcal{F}_{\bfs d_s}-(r-s+1)}.\end{align*}
By \eqref{eq: expression B bound degree}, Lemma \ref{lemma:
dimension B_0} and this bound, we obtain
$$|B_2(\fq)|\leq \Bigg(\bigg(\sum_{i=1}^sC_2(d_i,r,s)\bigg)^{r-s+1}+s+
\bigg(\sum_{i=1}^s\frac{\delta}{d_i}\bigg)^{r-s+1}\Bigg)q^{\dim
\mathcal{F}_{\bfs d_s}-r+s-1}.$$
To conclude, we observe that it suffices to show that
$$s\le 2\,s\,\delta\,\sigma^2-\delta\,\sigma^2\sum_{i=1}^s\Big(\frac{1}{d_i}+\frac{2}{\sigma}\Big)
-\delta\sum_{i=1}^s\frac{1}{d_i}.$$
Taking into account that $d_i\ge 2$ for each $i$ and thus $\sigma\ge
s$, the theorem readily follows.
\end{proof}

Finally, we express Theorem \ref{th: number systems not AI} in terms
of probabilities.
\begin{corollary}\label{coro: number systems not AI}
With notations as in Theorem \ref{th: number systems not AI},
considering the uniform probability in $\mathcal{F}_{\bfs d_s}$, the
probability ${\sf P}_2$ of the set of $\bfs F_s$ in
$\mathcal{F}_{\bfs d_s}$ for which $V(\bfs F_s)$ is not absolutely
irreducible is bounded in the following way:
$${\sf P}_2\le  \bigg(\frac{2\,s\,\sigma^2\delta}{q}\bigg)^{r-s+1}.$$
\end{corollary}

%
%----------------------------------------------------------------------------
%----------------------------------------------------------------------------
%----------------------------------------------------------------------------
%----------------------------------------------------------------------------
%
%----------------------------------------------------------------------------
%----------------------------------------------------------------------------
%----------------------------------------------------------------------------
%----------------------------------------------------------------------------
%
\bibliographystyle{amsalpha}

\bibliography{refs1, finite_fields}

\end{document}